\documentclass[11pt]{article}

\usepackage{authblk}
\usepackage{booktabs}
\usepackage[ruled, vlined]{algorithm2e}
\usepackage{amssymb}
\usepackage{setspace} 
\usepackage{bm}
\usepackage{subfiles}  
\usepackage{paralist}
\usepackage{amsmath,graphicx,amsthm,amssymb,xcolor}
\usepackage{paralist}
\usepackage{multirow}
\usepackage{makecell}
\usepackage{hyperref}
\usepackage[capitalize,noabbrev]{cleveref}
\hypersetup{colorlinks={true},linkcolor={magenta},citecolor={blue}}

\usepackage[left=1in,top=1in,right=1in,bottom=1in,letterpaper]{geometry}

\newcommand{\blue}[1]{{#1}}

\newcommand{\orange}[1]{{#1}}

\makeatletter
\def\@fnsymbol#1{\ensuremath{\ifcase#1\or \dagger\or \ddagger\or
		\mathsection\or \mathparagraph\or \|\or **\or \dagger\dagger
		\or \ddagger\ddagger \else\@ctrerr\fi}}
\makeatother

\newcommand*\samethanks[1][\value{footnote}]{\footnotemark[#1]}

\begin{document}

\title{Online Primal-Dual Algorithms For Stochastic Resource Allocation Problems}

\author[1]{Yuwei Chen \thanks{Equal contribution.}}
\author[1]{Zengde Deng \samethanks }
\author[1]{Zaiyi Chen}
\author[1]{Yinzhi Zhou}
\author[1]{Yujie Chen}
\author[1]{Haoyuan Hu}
\affil[1]{Cainiao Network, Hangzhou, Zhejiang, China}

\date{}
\maketitle

%\affiliation[inst2]{organization={Department Two},%Department and Organization
%            addressline={Address Two}, 
%            city={City Two},
%            postcode={22222}, 
%            state={State Two},
%            country={Country Two}}

\begin{abstract}

This paper studies the online stochastic resource allocation problem~(RAP) with chance constraints and conditional expectation constraints. 
The online RAP is an integer linear programming problem where resource consumption coefficients are revealed column by column along with the corresponding revenue coefficients. 
When a column is revealed, the corresponding decision variables are determined instantaneously without future information.
In online applications, the resource consumption coefficients are often obtained by prediction.
An application for such scenario rises from the online order fulfilment task. When the timeliness constraints are considered, the coefficients are generated by the prediction for the transportation time from origin to destination. To model their uncertainties, we take the chance constraints and conditional expectation constraints into the consideration.
%To the best of our knowledge, this is the first time chance constraints and conditional expectation constraints are introduced in the online RAP problem.
Assuming that the uncertain variables have known Gaussian distributions, the stochastic RAP can be transformed into a deterministic but nonlinear problem with integer second-order cone constraints. 
Next, we linearize this nonlinear problem and theoretically analyze the performance of vanilla online primal-dual algorithm for solving the linearized stochastic RAP. Under mild technical assumptions, the optimality gap and constraint violation are both on the order of $\sqrt{n}$.
Then, to further improve the performance of the algorithm, several modified online primal-dual algorithms with heuristic corrections are proposed. 
Finally, extensive numerical experiments demonstrate the applicability and effectiveness of our methods.
\end{abstract}



\section{Introduction}
\label{sec:intro}
\subsection{Background}
The resource allocation problem~(RAP) \cite{asadpour2020online} is to find the best allocation of a fixed amount of resources to various activities, in order to maximize the total revenue.
The online RAP has a wide range of applications such as network routing \cite{buchbinder2009online}, computer resource allocation \cite{kurose1989microeconomic}, Internet advertising \cite{mehta2007adwords}, portfolio selection \cite{ida2003portfolio} and so on. \blue{This paper studies a multi-dimensional online RAP with uncertainty.
There are $m$ resources and $k$ resource consumption schemes for each request. The request for the resources arrives one by one. When the $i$-th request is revealed, one or none of $k$ resource consumption schemes is chosen to satisfy this request. 
If $l$-th resource consumption scheme is chosen, the revenue and the consumption of the $j$-th resource are $c_{tl}$ and $a_{tjl}$ respectively. 
The decision is irrevocable and has to be decided immediately according to the historical information $\{(\bm c_\tau, \bm A_\tau)\}_{\tau=1}^t$, without future information. 
Our aim is to maximize the total revenue with limited resource capacities, given the total number $n$ of incoming requests and considering the uncertainty of $a_{tjl}$.}

\blue{This paper models the uncertainty with chance constraints and conditional expectation constraints.}
On one hand, a number of studies such as \cite{lu2021non, zhang2015resource, xu2016energy, cooper2004chance} have applied the offline chance constrained resource allocation model into engineering fields. \blue{Some of these fields are better suited for applying online models. Taking \cite{lu2021non} as an example, Lu et al. studied a offline planning model for selecting non-profit operations in a hospital considering the uncertainty of resource consumption. It is an online problem in practice since the decision must be made immediately when a operation request is arrived. 
On the other hand, the online chance constrained resource allocation model can be applied to the order fulfillment task, which is an important application in the field of the supply chain. 
In the order fulfillment problem, each order needs to be allocated to a transportation channel which will be delivered from the origin to the destination \cite{ricker1999order}. The goal is to maximize the total revenue while ensuring the average transportation time does not suppress the given threshold. We refer \orange{this average transportation time constraint} as the timeliness constraint. \orange{A long delivery time highly impairs the customer's experience. Consequently, timeliness constraint upper bound are often adopted to guarantee the relatively low average transportation time.} Due to the complicacy of the transportation process, the transportation time of each channel is attained by prediction and has uncertainty.
% Specifically, this uncertainty mainly comes from the prediction of resource consumption. 
% Moreover, the effectiveness of our proposed algorithms will be verified with real data obtained from Cainiao Network, a supply chain company.
}\orange{As far as we know, this is the first work to take both chance constraints and conditional expectation constraints into the consideration for the online RAP.}

\subsection{Related Works}
The deterministic RAP can be modeled as an integer linear programming~(ILP) problem. 
Many recent papers have studied the online ILP problems (see \cite{molinaro2014geometry, li2019online, agrawal2014dynamic, gupta2014experts, chen2015dynamic, balseiro2020best, gao2021boosting, li2020simple} and references therein). \blue{Algorithms in \cite{molinaro2014geometry, li2019online, agrawal2014dynamic, gupta2014experts, chen2015dynamic, balseiro2020best, gao2021boosting, li2020simple} are all dual-based which maintain dual prices in iterations and can achieve near-optimal solutions under mild conditions. 
When a new request arrives, the decision is made immediately based on the dual price vector.
Among these studies, researchers \cite{molinaro2014geometry, li2019online, agrawal2014dynamic, gupta2014experts, chen2015dynamic} construct dual problems by using historical information and solving them to obtain the dual prices. To deal with the disadvantage that solving dual problems may be time-consuming, researchers \cite{balseiro2020best, gao2021boosting, li2020simple} 
propose online primal-dual~(OPD) algorithms that update the dual prices by utilizing the dual mirror descent or projected stochastic subgradient descent.
It is worth to mention that these OPD algorithms are simple and efficient since the dual prices are updated through algebraic computation without solving optimization problems. 
% The algorithm in \cite{li2020simple} serves as the basis of our algorithms.
}

However, the optimization models studied in \cite{molinaro2014geometry, li2019online, agrawal2014dynamic,gupta2014experts, chen2015dynamic, gao2021boosting, kesselheim2014primal,li2020simple,balseiro2020best} are deterministic and may suffer from poor performance when the resource consumption is uncertain in practice.
In the existing articles that study the uncertain online optimization, the uncertainty is modeled by the worst-case scenario value, expectation, regret, or a linear combination of the above (see \cite{bent2004online, bent2005online, jiang2020online, liu2015averaging, jiang2020online2}).
These modeling methods are mainly aimed at the uncertainty in the objective function, while almost no chance constraint or conditional expectation constraint is considered in the existing studies.

Chance constrained programming \cite{charnes1959chance} is a widely used stochastic programming technique \blue{to model the uncertainty in constraints}.
In stochastic programming \cite{ruszczynski2003stochastic}, it is assumed that some parameters are uncertain and their distributions are known.
If the uncertain parameters in an active inequality constraint are set to the medians, the probability of this constraint not holding is 50\%. 
To avoid this issue in the online RAP, this paper adopts the chance constraints to model the uncertainty. The chance constraint is the constraint on the uncertain parameters whose holding probability is not lower than the prescribed level. 
The solution methods for chance constrained programming~(CCP) have been studied by \cite{charnes1959chance, jagannathan1974chance, charnes1963deterministic, prekopa1973contributions}. 
If the uncertain parameters have a known multivariate Gaussian distribution, the chance constrained counterparts of linear constraints can be transformed into deterministic second-order cone (SOC) constraints. 

The conditional expectation constraint \cite{prekopa1973contributions} can be served as a supplement to chance constraints, since the chance constraint does not restrict the amount of the violation of inequality constraints directly. In some cases, the expected amount of violation can be large even if the chance constraints are satisfied.
To address this concern, the conditional expectation constraint is also taken into consideration in this paper.
Similar to chance constraints, the conditional expectation constraints can also be transformed into deterministic SOC constraints under the assumption of multivariate Gaussian distribution \cite{prekopa1973contributions}. 

\subsection{Main Contributions}
This paper studies the online stochastic RAP considering the uncertainty of resource consumption coefficients. 
The chance constraint and conditional expectation constraint are used to model the uncertainty and both of them can be transformed into the SOC constraints equivalently.
The non-linearity and indecomposability of the SOC constraints make the online problem challenging to handle.
The main contributions of this paper are as follows.

\begin{enumerate}
	\item[(1)] To the best of our knowledge, this is the first time chance constraints and conditional expectation constraints are introduced in the online RAP. A linearization method is presented to transform the SOC constrained problem into a linear form suitable for the online solution.
	\item[(2)] We theoretically analyze the performance of the vanilla OPD algorithm when it is applied to solve the linearized stochastic RAP. Under mild technical assumptions, the expected optimality gap and constraint violation are both $O(\sqrt{n})$.
	\item[(3)] We propose modified versions of the OPD approach by leveraging the structure of the SOC constraints to effectively reduce the probability deviation of chance constraints and the constraint violation of conditional expectation constraints \blue{in practice}.
	\item[(4)] Massive numerical experiments are conducted to demonstrate the applicability and effectiveness of the proposed algorithms.
\end{enumerate} 

 The rest of this paper is organized as follows. 
 Section \ref{model} formulates the stochastic programming model of RAP and its linear relaxation.
 Section \ref{algorithm} introduces the proposed heuristic corrections and modified OPD algorithms for solving stochastic RAPs.
 Section \ref{exper} gives the numerical experiment results. 
 Finally in Section \ref{conclusion}, conclusions are drawn.

\section{Model Description}
\label{model}

This section first formulates the deterministic model of the RAP. 
Then, a stochastic counterpart with chance constraints and conditional expectation constraints is established. Finally, the stochastic programming problem is relaxed into an integer linear problem suitable for online solution.

\subsection{Deterministic Problem}
Consider the multi-dimensional RAP with $n$ requests and $m$ resources.
For each request, there are always $k$ resource consumption schemes that can satisfy it. 
When a request is revealed, the decision maker chooses one scheme or none.
Without loss of generality, a deterministic multi-dimensional RAP can be modeled as follows:

\begin{equation}\label{prob:deterministic}
	\begin{aligned}
		\max_{\bm{x}}\quad &\sum_{t=1}^n \bm{c}_t^{\top} \bm{x}_t \\
		\rm{s.t.}\quad&\sum_{t=1}^n \bm{a}_{tj}^{\top} \bm{x}_t \le b_j, \forall j = 1,\dots,m\\
		&\bm{1}^{\top} \bm{x}_t \le 1, \bm{x}_t \in \{0,1\}^k, \forall t = 1,\dots,n
	\end{aligned}
\end{equation}
where the revenue coefficient vector $\bm{c}_t \in \mathbb{R}^k$, and the resource consumption vector $\bm{a}_{tj}\in \mathbb{R}^k$. The decision variables are $\bm{x} = (\bm{x}_1,\dots,\bm{x}_n)^\top$. 
$x_{tl} = 1$ means that $t$-th request is satisfied by resource consumption scheme $l$.
$b_j$ is the capacity of resource $j$. $\bm{1}$ denotes all-one vector.
In the online setting of ILP, the input data $(\bm{c}_t,\bm{a}_{t1},\dots,\bm{a}_{tm})$ is revealed one by one and $\bm{x}_t$ is determined instantaneously when $(\bm{c}_t,\bm{a}_{t1},\dots,\bm{a}_{tm})$ is revealed.

Throughout this paper, we assume that the input data of deterministic problem and stochastic programming problem is drawn i.i.d.~from some distributions. Moreover, $n$ and $\bm b$ are assumed to be known and fixed.

\subsection{Stochastic Programming Problem}
In practice, the value of $\bm{a}_{tj}$ may be obtained by predicting that yields the uncertainty.
Consequently, taken the uncertainty of $\bm{a}_{tj}$ into consideration, we formulate the following chance constraints:
\begin{equation}\label{eq:cc}
	\mathbb{P}_{\bm a_{tj}\sim \mathcal P_{a}}\left(\sum_{t=1}^n \bm{a}_{tj}^{\top} \bm{x}_t \le b_j\right)\ge \eta_j,\forall j = 1,\dots,m,
\end{equation}
where $\mathbb{P}$ means probability, $\eta_j$ is the given confidence level and $\mathcal P_a$ is the distribution of $\bm a_{tj}$.

The constraints \eqref{eq:cc} only restrict the probability of violating inequality and do not restrict the amount of violation directly. 
% In addition to (\ref{eq:cc}), the constraints (\ref{eq:ce}) involving conditional expectations are also studied in this paper to handle the uncertainty.
\blue{To address this issue, the conditional expectation constraints \eqref{eq:ce} are also considered in this paper which can limit the expected violation of these inequality constraints.}
\begin{equation}\label{eq:ce}
	\mathbb{E}_{\bm a_{tj}\sim \mathcal P_{a}}\left[\sum_{t=1}^n\bm{a}_{tj}^{\top} \bm{x}_t - b_j\bigg|\sum_{t=1}^n\bm{a}_{tj}^{\top} \bm{x}_t - b_j > 0\right]\le \gamma_j,\forall j = 1,\dots,m,
\end{equation}
where $\mathbb E$ denotes expectation and $
\gamma_j$ is the given parameter. This paper assumes that the expectation in \eqref{eq:ce} always exists.

% Combining the constraints (\ref{eq:cc}), constraints (\ref{eq:ce}) and deterministic problem (\ref{prob:deterministic}), 
\blue{Replacing the deterministic constraints in \eqref{prob:deterministic} by the chance constraints \eqref{eq:cc} and conditional expectation constraints \eqref{eq:ce}},
we formulate the following stochastic programming problem:\vspace{-0.1cm}
\begin{equation}\label{OriginalCCP}
	\begin{aligned}
		\max_{\bm{x}}\quad &\sum_{t=1}^n \bm{c}_t^{\top} \bm{x}_t \\
		\rm{s.t.}\quad& \mathbb{P}_{\bm a_{tj}\sim \mathcal P_{a}}\left(\sum_{t=1}^n \bm{a}_{tj}^{\top} \bm{x}_t \le b_j\right)\ge \eta_j,\forall j = 1,\dots,m\\
		&\mathbb{E}_{\bm a_{tj}\sim \mathcal P_{a}}\left[\sum_{t=1}^n\bm{a}_{tj}^{\top} \bm{x}_t-b_j\bigg|\sum_{t=1}^n\bm{a}_{tj}^{\top} \bm{x}_t - b_j > 0\right]\le \gamma_j,\forall j = 1,\dots,m\\
		&\bm{1}^{\top} \bm{x}_t \le 1, \bm{x}_t \in \{0,1\}^k, \forall t = 1,\dots,n.
	\end{aligned}
\end{equation}
\vspace{-0.15cm}For convenience, we refer to the problem \eqref{OriginalCCP} as a CCP problem throughout this paper, thought \eqref{eq:ce} are not standard chance constraints.

\subsection{Equivalent Transformation}
In the online setting of the CCP problem \eqref{OriginalCCP}, the revenue vector $\bm c_t$ and the distributions of the resource consumption vectors $\{\bm a_{t1},\dots,\bm a_{tm}\}$ are revealed one by one. For any $t$ and $j$, we assume that $\bm a_{tj}\sim N(\bar{\bm a}_{tj}, \bm K_{tj})$ which means the true value of $\bm{a}_{tj}$ belongs to a known Gaussian distribution with mean $\bar{\bm{a}}_{tj}$ and covariance matrix $\bm{K}_{tj}$. 
In the following, we will reformulate \eqref{OriginalCCP} into a deterministic problem.

First is to reformulate the chance constraints \eqref{eq:cc}. According to the \blue{properties of Gaussian distribution, we have that} $\sum_{t=1}^n \bm{a}_{tj}^\top \bm{x}_t \sim N(\sum_{t=1}^n \bar{\bm{a}}_{tj}^\top \bm{x}_t$, $\sum_{t=1}^n \bm{x}_t^\top \bm{K}_{tj} \bm{x}_t)$. Thus, the constraints \eqref{eq:cc} are equivalent to the following constraints \cite{jagannathan1974chance}:
\begin{equation}\label{eq:rcc}
	\sum_{t=1}^n \bar{\bm{a}}_{tj}^\top \bm{x}_t + \Phi^{-1}(\eta_j)\sqrt{\sum_{t=1}^n \bm{x}_t^\top \bm{K}_{tj} \bm{x}_t}\le b_j,\forall j = 1,\dots,m,
\end{equation}
where $\Phi(\cdot)$ represents the cumulative distribution function of the standard Gaussian
distribution. 

Next is to reformulate the conditional expectation constraints \eqref{eq:ce}. 
The value of the conditional expectation $\mathbb{E}[\sum_{t=1}^n\bm{a}_{tj}^{\top} \bm{x}_t-b_j|\sum_{t=1}^n\bm{a}_{tj}^{\top} \bm{x}_t - b_j > 0]$ is related to the variance $\sqrt{\sum_{t=1}^n \bm{x}_t^\top \bm{K}_{tj} \bm{x}_t}$, which makes the parameter $\gamma_j$ difficult to quantify. As an alternative, we replace $\gamma_j$ with the normalized value $\tilde{\gamma}_j$ that equals to $\gamma_j/\sqrt{\sum_{t=1}^n \bm{x}_t^\top \bm{K}_{tj} \bm{x}_t}$. Specifically, the constraint \eqref{eq:ce} is rewritten as
\begin{equation}\label{eq:ce2}
	\mathbb{E}_{\bm a_{tj}\sim \mathcal P_{a}}\left[\frac{\sum_{t=1}^n\bm{a}_{tj}^{\top} \bm{x}_t - b_j}{\sqrt{\sum_{t=1}^n \bm{x}_t^\top \bm{K}_{tj} \bm{x}_t}}\bigg|\sum_{t=1}^n\bm{a}_{tj}^{\top} \bm{x}_t - b_j > 0\right]\le \tilde{\gamma}_j,\forall j = 1,\dots,m,
\end{equation}

According to \cite{prekopa1973contributions}, the constraints \eqref{eq:ce2} are equivalent to 
\begin{equation}\label{eq:rce}
	\sum_{t=1}^n \bar{\bm{a}}_{tj}^\top \bm{x}_t + h^{-1}(\tilde{\gamma_j})\sqrt{\sum_{t=1}^n \bm{x}_t^\top \bm{K}_{tj} \bm{x}_t} \le b_j,\forall j = 1,\dots,m,
\end{equation}
where the function $h(\cdot)$ has the form\footnote{\blue{In \cite{prekopa1973contributions}, $\sqrt{\cdot}$ in $\sqrt{2\pi}$ is left out which is a typo.}}
\begin{equation}
	h(z) = \frac{e^{-\frac{1}{2}z^2}/\sqrt{2\pi}}{1-\Phi(z)}-z.
\end{equation}

% Under the assumption of Gaussian distribution, the forms of constraints \eqref{eq:rcc} and \eqref{eq:rce} are the same. They can be combined into
\blue{Under the assumption of Gaussian distribution, the constraints \eqref{eq:rcc} and \eqref{eq:rce} have the same formulations and the only difference lies on the coefficient of $\sqrt{\sum_{t=1}^n \bm{x}_t^\top \bm{K}_{tj} \bm{x}_t}$. Consequently, when both \eqref{eq:rcc} and \eqref{eq:rce} exist, we can merge them into the following constraint,}
\begin{equation}\label{eq:rccc}
	\sum_{t=1}^n \bar{\bm{a}}_{tj}^\top \bm{x}_t + \psi_j\sqrt{\sum_{t=1}^n \bm{x}_t^\top \bm{K}_{tj} \bm{x}_t} \le b_j,\forall j = 1,\dots,m,
\end{equation}
where $\psi_j = \max\{\Phi^{-1}(\eta_j),h^{-1}(\tilde{\gamma}_j)\}$. When $\eta_j > 50\%$ or $\tilde{\gamma}_j < \sqrt{\pi/2}$, $\psi_j > 0$.

Substituting \eqref{eq:rccc} into \eqref{OriginalCCP}, the CCP problem \eqref{OriginalCCP} is equivalent to the following deterministic problem:

\begin{equation}\label{NLCCP}
	\begin{aligned}
		\max_{\bm{x}}\quad &\sum_{t=1}^n \bm{c}_t^\top \bm{x}_t \\
		\rm{s.t.}\quad  &\sum_{t=1}^n \bar{\bm{a}}_{tj}^\top \bm{x}_t + \psi_j\sqrt{\sum_{t=1}^n \bm{x}_t^\top \bm{K}_{tj} \bm{x}_t}\le b_j, \forall j = 1,\dots,m.\\
		&\bm{1}^{\top} \bm{x}_t \le 1, \bm{x}_t \in \{0,1\}^k, \forall t = 1,\dots,n
	\end{aligned}
\end{equation}

The problem \eqref{NLCCP} is an integer SOC programming~(ISOCP) problem when $\psi_j>0,\forall j$. 
The offline version of \eqref{NLCCP} can be solved by commercial solvers such as Gurobi. 
However, in the online setting, it is difficult to solve problem \eqref{NLCCP} due to its \blue{indecomposability}: $\bm{x}_t$ with different subscripts $t$ are coupled with each other in $\sqrt{\sum_{t=1}^n \bm{x}_t^\top \bm{K}_{tj} \bm{x}_t}$. \blue{In other words, calculating the individual constraint consumption at each time step $t$ is challenging.} \blue{Moreover, we make the following assumption throughout this paper.}\vspace{-0.1cm}
\newtheorem{assumption}{\bf Assumption}[]
\begin{assumption}\label{assumption 1}
	We assume
	\begin{enumerate}
		\item [(a)] The coefficient set $\{c_{tj}, \bm{\bar{a}}_{tj}, \bm K_{tj}\}$'s are i.i.d.~sampled from an unknown distribution $\mathcal P$.
		\item [(b)] The coefficient set $\{c_{tj}, \bm{\bar{a}}_{tj}, \bm K_{tj}\}$'s are bounded.
		\item [(c)] The right-hand-side $\bm b = n\bm d$. $\bm d$ is bounded and its upper and lower bounds are both positive.
	\end{enumerate}
\end{assumption}

\subsection{Relaxed Linear Problem} \label{subsec:relaxed_lp}
\blue{To begin, we provide the following proposition.}
\newtheorem{prop}{Proposition}[]
\begin{prop}\label{prop:1}
For arbitrary $t$ and $j$, the following equation holds.\vspace{-0.1cm}
\begin{equation*} 
	\sqrt{\bm{x}_t^\top \bm{K}_{tj} \bm{x}_t} = \bm{\bm{\gamma}}^\top_{tj} \bm{x}_t,\forall \bm{x}_t\in\{\bm{x} \in \{0,1\}^k|\bm{1}^\top\bm{x} \le 1\}, \vspace{-0.3cm}
\end{equation*}
where $\bm{\gamma}_{tj}$ is formed by concatenating the square roots of the diagonal elements of the matrix $\bm{K}_{tj}$.
\end{prop}\vspace{-0.3cm}
\begin{proof}
See \ref{sec:appendix prop 1}.
\end{proof}
To address the non-decomposable issue raised by the non-linearity of $\sqrt{\sum_{t=1}^n \bm{x}_t^\top \bm{K}_{tj} \bm{x}_t}$, we linearize this term to decouple different $\bm{x}_t$.
Specifically, according to Cauchy-Schwarz inequality
$\sqrt{n\sum_{t=1}^n\bm{x}_t^\top \bm{K}_{tj} \bm{x}_t} \ge \sum_{t=1}^n $ 
$\sqrt{\bm{x}_t^\top \bm{K}_{tj} \bm{x}_t}$
and Proposition~\ref{prop:1}, the nonlinear \blue{and non-decomposable} problem \eqref{NLCCP} can be approximated by
\vspace{-0.2cm}
\begin{equation} \label{RelaxedCCP}
	\begin{aligned}
		\max_{\bm{x}}\quad &\sum_{t = 1}^n \bm{c}_t^\top \bm{x}_t \\
		\rm{s.t.}\quad&\sum_{t = 1}^n \left(\bar{\bm{a}}_{tj}^\top + \frac{\psi_j}{\sqrt{n}} \bm{\gamma}^\top_{tj}  \right) \bm{x}_t \le b_j, \forall j = 1,\dots,m\\
		&\bm{1}^{\top} \bm{x}_t \le 1, \bm{x}_t \in \{0,1\}^k, \forall t = 1,\dots,n.
	\end{aligned}\vspace{-0.1cm}
\end{equation}

\blue{It is worth to mention that the relaxed} problem \eqref{RelaxedCCP} is an integer LP~(ILP) problem \blue{which is linear and decomposable}, and can be solved in the online setting by existing algorithms. 
The vanilla OPD algorithm for solving this relaxed problem is the basis of our algorithms for solving the CCP problem \eqref{NLCCP} and we will detail it in the following section.

% \begin{proof}
% In case $\bm{x}_t = 0$, the equation obviously holds. $\forall \bm{x}_t = (0,\dots,0,\overset{l\text{th}}{1},0,\dots,0)^\top,l=1,\dots,k$, we have
% \begin{equation*}
%     \sqrt{\bm{x}_t^\top \bm{K}_{tj} \bm{x}_t} = \sqrt{K_{tjll}}
%     = \gamma_{tjl} = \bm{\gamma}_{tj}^\top \bm{x}_t.
% \end{equation*}
% The proof is completed.
% \end{proof}

\section{Solution Algorithms}
\label{algorithm}

% In this section, we present OPD algorithms for solving the online second-order cone constrained problem \eqref{NLCCP}, which is derived from the vanilla OPD algorithm for solving \eqref{RelaxedCCP} and several heuristic correction methods.
\blue{In this section, we introduce several online primal-dual methods to handle the online SOC constrained problem \eqref{NLCCP}. Firstly, we revisit the state-of-the-art OPD algorithm for solving the relaxed problem \eqref{RelaxedCCP}. Then, some heuristic correction methods based on the structure of \eqref{NLCCP} are proposed to improve the practical performance.}

\vspace{-0.2cm}
\subsection{OPD Algorithm for online ILP}
\blue{Recall that \eqref{RelaxedCCP} is an ILP problem} and Li et al.~\cite{li2020simple} have proposed an effective OPD algorithm \blue{to solve the online ILP problem}. 
% Define $\tilde{\bm{a}}_{tj} = \bar{\bm{a}}_{tj} + {\psi_j}\bm{\gamma}_{tj}{/\sqrt{n}}$, then the ILP problem (\ref{RelaxedCCP}) can be considered as a standard deterministic RAP and its online version can be solved by Algorithm \ref{Algo 1}. 
\blue{For simplicity, denote $\tilde{\bm{a}}_{tj} = \bar{\bm{a}}_{tj} + {\psi_j}\bm{\gamma}_{tj}{/\sqrt{n}}$ and we present the OPD method as shown in Algorithm \ref{Algo 1}.}

\begin{algorithm}[htb]\label{Algo 1}
%\SetAlgoNoLine
\caption{OPD Algorithm for ILP}
\LinesNumbered
\KwIn{$\bm{d} = \bm{b}/n$}
\KwOut{$\bm{x}=(\bm{x}_1,...,\bm{x}_n)$}
{\bf Initialize:} $\bm{p}_1=\bm{0}$\\
\For{$t=1,...,n$}
{
	Set $v_t = \max_{l=1,\dots,k} \  (\bm{c}_t^\top - \bm{p}_t^\top \tilde{\bm{A}}_{t})\bm{e}_l$\\
	\eIf{$v_t > 0$}
	{
		Pick an index $l_t$ randomly from
		\vspace{-0.1cm}
 $$\left\{l:v_t=(\bm{c}_t^\top - \bm{p}_t^\top \tilde{\bm{A}}_{t})\bm{e}_l\right\}$$\\	
 \vspace{-0.1cm}
  		Set $\bm{x}_t = \bm{e}_{l_t}$
%		Set $$x_{tl}=\begin{cases}
%			1,\ & l = l_t\\
%			0,\ & \text{otherwise}
%		\end{cases}$$\\
	}{Set $\bm{x}_t = \bm{0}$}
	Compute {
	\vspace{-0.2cm}
	$$\bm{p}_{t+1} = \max\left\{\bm{p}_{t} + \frac{1}{\sqrt{n}}\left(\tilde{\bm{A}}_{t}\bm{x}_t-\bm{d}\right), \mathbf{0}\right\}$$
	\vspace{-0.1cm}
	}
}
\end{algorithm}

In Algorithm \ref{Algo 1}, we denote $\tilde{\bm{A}}_t = (\tilde{\bm{a}}_{t1}^\top,\dots,\tilde{\bm{a}}_{tm}^\top)^\top$, $\bm{b} = (b_1,\dots,b_m)^\top$ and $\bm{x}_t = (x_{t1},\dots,x_{tk})^\top$. 
Algorithm \ref{Algo 1} is a dual-based algorithm which maintains a dual vector $\bm{p}_t$. 
In each round $t$, $\bm{c}_t$ and $\tilde{\bm{A}}_{t}$ are revealed. 
Then, $\bm{x}_t$ is determined immediately by choosing $l$ that maximizes $(\bm{c}_t^\top - \bm{p}_t^\top \tilde{\bm{A}}_{t})\bm{e}_l$, where $\bm e_l \in \mathbb R^k$ and $\bm{e}_l$ is the unit vector \blue{with $l$-th coordinate being $1$}. 
{After determining $\bm{x}_t$, $\bm{p}_t$ is updated by a projected stochastic subgradient descent method where $(\bm{d} - \tilde{\bm{A}}_{t}\bm{x}_t)$ is the subgradient corresponding to $\bm{p}_t$.}
% In Line 9, $\bm{p}_{t+1}\lor0 = (\max\{p_{(t+1)1},0\},\dots,\max\{p_{(t+1)m},0\})^\top$.
Li et al.~\cite{li2020simple} have shown that Algorithm \ref{Algo 1} provides a near-optimal solution of the problem \eqref{Linear-relaxedCCP} that is the linear relaxation of \eqref{RelaxedCCP}, as stated in Theorem \ref{thm1}.
\vspace{-0.1cm}
\begin{equation} \label{Linear-relaxedCCP}
	\begin{aligned}
		\max_{\bm{x}}\quad &\sum_{t = 1}^n \bm{c}_t^\top \bm{x}_t \\
		\rm{s.t.}\quad&\sum_{t = 1}^n \left(\bar{\bm{a}}_{tj}^\top + \frac{\psi_j}{\sqrt{n}} \bm{\gamma}^\top_{tj}  \right) \bm{x}_t \le b_j,\forall j = 1,\dots,m\\
		&\boldsymbol{1}^{\top} \boldsymbol{x}_t \le 1, \bm{x}_t \ge \bm{0}, \forall t = 1,\dots,n.
	\end{aligned}
\end{equation}

\newtheorem{thm}{\bf Theorem}[]
\begin{thm}[Theorem 3 in \cite{li2020simple}]\label{thm1}
%Suppose $(\bm{c}_t,\tilde{\bm{A}}_t)$ is i.i.d. and bounded, and the upper and lower bounds of $\bm{b}$ are finite and positive. 
Under Assumption \ref{assumption 1}, the upper bound of the expected optimality gap \eqref{eq:regret} and the expected constraint violation \eqref{eq:violation} of Algorithm \ref{Algo 1} compared to the optimal solution of the LP problem \eqref{Linear-relaxedCCP} are on the order of $\sqrt{n}$, i.e., 
\vspace{-0.1cm}
\begin{equation}\label{eq:regret}
    {\mathbb E_{\xi_{tj} \sim \mathcal P}} \left[\hat{R}_n^{LP}-\sum_{t = 1}^n \bm{c}_t^\top \bm{x}_t\right] \le O(\sqrt{n})
\end{equation}
% \newpage

\begin{equation}\label{eq:violation}
    {\mathbb E_{\xi_{tj} \sim \mathcal P}}\left[\left\|\left(\sum_{t=1}^n \tilde{\bm{A}}_{t} \bm{x}_t-\bm{b}\right)^+\right\|_2\right] \le O(\sqrt{n})
\end{equation}
where $\hat{R}_n^{LP}$ is the optimal objective value of the LP problem \eqref{Linear-relaxedCCP} and $\bm{x}_t$ is the output of Algorithm \ref{Algo 1} and $(\cdot)^+$ is the positive part function. \blue{$\xi_{tj}$ denotes the coefficient set $\{c_{tj}, \bm{\bar{a}}_{tj}, \bm K_{tj}\}$ and $\mathcal P$ is any distribution that satisfies Assumption \ref{assumption 1} (b).}
\end{thm}

\blue{Theorem \ref{thm1} states the upper bound of the expected optimality gap and the expected constraint violation, which are both $O(\sqrt{n})$. Then, we provide the lower bound of the expected optimality gap in the following theorem.}
 
\begin{thm}\label{thm: lower bound of OPD}
%Suppose $(\bm{c}_t,\tilde{\bm{A}}_t)$ is i.i.d. and bounded, and the upper and lower bounds of $\bm{b}$ are finite and positive. 
Under Assumption \ref{assumption 1}, the lower bound of the expected optimality gap of Algorithm \ref{Algo 1} compared to the optimal solution of the LP problem \eqref{Linear-relaxedCCP} is on the order of $\sqrt{n}$, i.e., \vspace{-0.1cm}
\begin{equation}\label{eq:lower bound opt gap}
    {\mathbb E}_{\xi_{tj} \sim \mathcal P} \left[\hat{R}_n^{LP}-\sum_{t = 1}^n \bm{c}_t^\top \bm{x}_t\right] \ge -O(\sqrt{n})
\end{equation}
where 
$\hat{R}_n^{LP}$ is the optimal objective value of the LP problem \eqref{Linear-relaxedCCP} and 
$\bm{x}_t$ is the output of Algorithm \ref{Algo 1}.
\end{thm}
\begin{proof}
	See \ref{sec:appendix thm lower bound opt gap}.
\end{proof}

Next, we will analysis the performance of Algorithm \ref{Algo 1} compared to the optimal solution of the ISOCP problem \eqref{NLCCP}. 
The linear relaxation of the ISOCP problem \eqref{NLCCP} is \vspace{-0.2cm}
\begin{equation}\label{Linear-NLCCP}
	\begin{aligned}
		\max_{\boldsymbol{x}}\quad &\sum_{t=1}^n \boldsymbol{c}_t^\top \boldsymbol{x}_t \\
		\rm{s.t.}\quad  &\sum_{t=1}^n \bar{\boldsymbol{a}}_{tj}^\top \boldsymbol{x}_t + \psi_j\sqrt{\sum_{t=1}^n \boldsymbol{x}_t^\top \bm{K}_{tj} \boldsymbol{x}_t}\le b_j, \forall j = 1,\dots,m\\
		&\boldsymbol{1}^{\top} \boldsymbol{x}_t \le 1, \bm{x}_t \ge \bm{0}, \forall t = 1,\dots,n
	\end{aligned}
\end{equation}
where the binary variables are relaxed into continuous variables on $[0, 1]$.
The optimal objective values of the ISOCP problem \eqref{NLCCP} and SOCP problem \eqref{Linear-NLCCP} are referred to as $\hat{R}_n^{ISOCP}$ and $\hat{R}_n^{SOCP}$, respectively.
\blue{Then, we establish the optimality gap between the LP problem \eqref{Linear-relaxedCCP} and SOCP problem \eqref{Linear-NLCCP}, and the constraint violation in the following lemma.}
% Lemma \ref{lemma1} states that the gap between the LP problem \eqref{Linear-relaxedCCP} and SOCP problem \eqref{Linear-NLCCP} is also on the order of $\sqrt{n}$ in terms of the regret and constraint violation. 
% Then, we get Theorem \ref{thm2} by combining Theorem \ref{thm1}, Theorem \ref{thm: lower bound of OPD} and Lemma \ref{lemma1}: Algorithm \ref{Algo 1} achieves $O(\sqrt{n})$ expected optimality gap and constraint violation compared to the optimal solution of the SOCP \eqref{Linear-NLCCP}. 
% Moreover, we have $\hat{R}_n^{ISOCP} \le \hat{R}_n^{SOCP}$ evidently. Thus, Algorithm \ref{Algo 1} achieves $O(\sqrt{n})$ regret and constraint violation compared to the optimal solution of the ISOCP \eqref{NLCCP} as stated in Theorem \ref{thm: regret and violation ISOCP}.

\newtheorem{lemma}{\bf Lemma}[]
\begin{lemma} \label{lemma1}
Under Assumption \ref{assumption 1}, we have the following results on the SOCP problem \eqref{Linear-NLCCP} and LP problem \eqref{Linear-relaxedCCP}.
\begin{enumerate}
	\item [(a)] If the decision variables $(\bm x_1^\top,\dots,\bm x_n^\top)^\top$ of the SOCP problem \eqref{Linear-NLCCP} is set to the optimal solution of the LP problem \eqref{Linear-relaxedCCP}, the optimality gap of \eqref{Linear-NLCCP} satisfies
	
	\begin{equation}
   		-O(\sqrt{n}) \le \hat{R}_n^{SOCP}-\sum_{t = 1}^n \bm{c}_t^\top \hat{\bm{x}}_t^{LP} \le 0
	\end{equation}
where $\hat{R}_n^{SOCP}$ is the optimal objective value of \eqref{Linear-NLCCP} and $\hat{\bm{x}}_t^{LP}$ is the optimal solution of \eqref{Linear-relaxedCCP}).
    \item [(b)] For any $\bm x = (\bm x_1^\top,\dots,\bm x_n^\top)^\top$ satisfying $\bm x_t \in \{\bm x \in \mathbb R^k|\bm 1^\top \bm x\le \bm 1, \bm x \ge \bm 0\},\forall t = 1,\dots,n$, the gap between the resource consumption of \eqref{Linear-NLCCP} and \eqref{Linear-relaxedCCP} satisfies
    
	\begin{equation}
	    \left\|\bm g\left({\bm{x}}\right)-\sum_{t=1}^n\tilde{\bm A}_t \bm x_t\right\|_2 \le O(\sqrt{n})
	\end{equation}
where $g_j(\bm{x}) = \sum_{t=1}^n \bar{\boldsymbol{a}}_{tj}^\top \bm{x}_t + \psi_j\sqrt{\sum_{t=1}^n \boldsymbol{x}_t^\top \bm{K}_{tj} \boldsymbol{x}_t}$ is the resource consumption of \eqref{Linear-NLCCP}.
\end{enumerate}
\end{lemma}
\begin{proof}
	See \ref{sec:appendix lemma1}.
\end{proof}
\blue{As shown in Lemma \ref{lemma1}, both the optimal gap and the constraint violation between the LP problem \eqref{Linear-relaxedCCP} and SOCP problem \eqref{Linear-NLCCP} are $O(\sqrt{n})$. Then, putting Theorem \ref{thm1}, Theorem \ref{thm: lower bound of OPD}, and Lemma \ref{lemma1} together, we can obtain that the expected optimality gap and constraint violation compared to the optimal solution of the SOCP problem \eqref{Linear-NLCCP} in Theorem \ref{thm2}.} 

\begin{thm}\label{thm2}
Under Assumption \ref{assumption 1}, the expected optimality gap and constraint violation of Algorithm \ref{Algo 1} compared to the optimal solution of the SOCP problem (\ref{Linear-NLCCP}) are on the order of $\sqrt{n}$, i.e., 

\begin{equation}\label{regret:SOCP}
    -O(\sqrt{n}) \le {\mathbb E_{\xi_{tj} \sim \mathcal P}} \left[\hat{R}_n^{SOCP}-\sum_{t = 1}^n \bm{c}_t^\top \bm{x}_t\right] \le O(\sqrt{n})
\end{equation}
\begin{equation}
    {\mathbb E_{\xi_{tj} \sim \mathcal P}} \left[\left\|\left(\bm g\left({{\bm{x}}}\right)-\bm{b}\right)^+\right\|_2\right] \le O(\sqrt{n})
\end{equation}
where 
$\hat{R}_n^{SOCP}$ is the optimal objective value of \eqref{Linear-NLCCP}, $\bm{x}_t$ is the output of Algorithm \ref{Algo 1} and $g(\bm x)$ is the left-hand side of the SOC constraints. 
\end{thm}
\begin{proof}
	See \ref{sec:appendix thm2}.
\end{proof}
\blue{Moreover, inequality $\hat{R}_n^{ISOCP} \le \hat{R}_n^{SOCP}$ holds due to the fact that the SOCP problem \eqref{Linear-NLCCP} is the linear relaxation of the ISOCP problem \eqref{NLCCP}.} Thus, Algorithm \ref{Algo 1} achieves $O(\sqrt{n})$ regret and constraint violation compared to the optimal solution of the ISOCP problem \eqref{NLCCP} as stated in Theorem \ref{thm: regret and violation ISOCP}.

\begin{thm}\label{thm: regret and violation ISOCP}
Under Assumption \ref{assumption 1}, the regret and constraint violation of Algorithm \ref{Algo 1} compared to the optimal solution of the ISOCP problem \eqref{NLCCP} are on the order of $\sqrt{n}$, i.e.,\vspace{-0.1cm}
\begin{equation}\label{regret:ISOCP}
   {\mathbb E_{\xi_{tj} \sim \mathcal P}} \left[\hat{R}_n^{ISOCP}-\sum_{t = 1}^n \bm{c}_t^\top \bm{x}_t\right] \le O(\sqrt{n})\vspace{-0.1cm}
\end{equation}
\begin{equation}
    {\mathbb E_{\xi_{tj} \sim \mathcal P}} \left[\left\|\left(\bm g\left({{\bm{x}}}\right)-\bm{b}\right)^+\right\|_2\right] \le O(\sqrt{n})\vspace{-0.1cm}
\end{equation}
where $\hat{R}_n^{ISOCP}$ is the optimal objective value of \eqref{NLCCP}, $\bm{x}_t$ is the output of Algorithm \ref{Algo 1} and $g(\bm x)$ is the left-hand side of the SOC constraints. 
\end{thm}

\subsection{Modified OPD Algorithms for Online CCP}
Although Algorithm \ref{Algo 1} has been able to obtain a near-optimal solution of the ISOCP problem \eqref{NLCCP} according to Theorem \ref{thm: regret and violation ISOCP}, its practical performance can be further improved \blue{by narrowing the gap between the solutions generated by Algorithm~\ref{Algo 1} and the offline ISOCP \eqref{NLCCP}}. 
\blue{To be specific,} this gap mainly comes from the following two points:
\begin{enumerate}
\renewcommand{\labelenumi}{(\alph{enumi})}
	\item The error between the offline ILP problem \eqref{RelaxedCCP} and the offline ISOCP problem \eqref{NLCCP}.
	\item The error between the online solution and offline solution of the ILP problem \eqref{RelaxedCCP}.
\end{enumerate}

\begin{algorithm}[tb]\label{Algo 2}
%\SetAlgoNoLine
\caption{Modified OPD Algorithm for CCP}
\LinesNumbered
\KwIn{$\bm{d} = \bm{b}/n$}
\KwOut{$\bm{x}=(\bm{x}_1,...,\bm{x}_n)$}
{\bf Initialize:} $\bm{p}_1=\bm{0}$\\
\For{$t=1,...,n$}
{
	Compute $\bm{\beta}_{t}$ via equation \eqref{beta}\\
	Set $v_t = \max_{l=1,\dots,k} \  (\bm{c}_t^\top - \bm{p}_t^\top \hat{\bm{A}}_{t}(\bm{\beta}_t))\bm{e}_l$\\
	\eIf{$v_t > 0$}
	{
		Pick an index $l_t$ randomly from \vspace{-0.15cm}
 $$\left\{l:v_t=(\bm{c}_t^\top - \bm{p}_t^\top \hat{\bm{A}}_{t}(\bm{\beta}_t))\bm{e}_l\right\}$$\\	\vspace{-0.15cm}
 		Set $\bm{x}_t = \bm{e}_{l_t}$
%		Set $$x_{tl}=\begin{cases}
%			1,\ & l = l_t\\
%			0,\ & \text{otherwise}
%		\end{cases}$$\\
	}{Set $\bm{x}_t = \bm{0}$}
	Compute {\vspace{-0.2cm}
	$$\ \bm{p}_{t+1} = \max\left\{\bm{p}_{t} + \frac{1}{\sqrt{n}}\left(\hat{\bm{A}}_{t}(\bm{\beta}_t)\bm{x}_t-\bm{d}\right),\mathbf{0}\right\}$$\vspace{-0.2cm}
	}
}
\end{algorithm}

To address these issues, we propose two modified OPD algorithms (Algorithms \ref{Algo 2} and \ref{Algo 3}) to solve the online CCP problem \eqref{NLCCP}.
% Online Algorithm \ref{Algo 2} is proposed in this paper for solving the CCP problem (\ref{NLCCP}). 
In Algorithm \ref{Algo 2}, a heuristic correction is applied to correct the error (a). \blue{For the $j$-th constraint,} we introduce scale factors \vspace{-0.15cm}
\begin{equation}\label{beta}
\beta_{tj}=
\begin{cases}
	1, &t = 1 \text{ or } \sum_{i=1}^{t-1} \bm\gamma_{ij}^\top \bm x_i = 0\\
	\frac{\sqrt{t-1}\sqrt{\sum_{i=1}^{t-1} \bm{x}_i^\top \bm{K}_{ij} \bm{x}_i}}
		{\sum_{i=1}^{t-1} \bm{\gamma}_{ij}^\top \bm{x}_i}, &t = 2,\dots,n \text{ and } \sum_{i=1}^{t-1} \bm\gamma_{ij}^\top \bm x_i > 0
\end{cases}\vspace{-0.15cm}
\end{equation}
to reduce the gap between $\sqrt{\sum_{t=1}^n \bm{x}_t^\top \bm{K}_{tj} \bm{x}_t}$ and $\sum_{t=1}^n {\bm{\gamma}_{tj}^\top \bm{x}_t}$ $/{\sqrt{n}}$.
\orange{Intuitively speaking, recalling that in section \ref{subsec:relaxed_lp} we utilize Cauchy-Schwarz inequality to linearize the non-decomposable problem \eqref{NLCCP}, equation \eqref{beta} can be regarded as the empirical correction for the non-decomposable term $\sqrt{(t-1)\sum_{i=1}^{t-1}\bm{x}_i^\top \bm{K}_{ij} \bm{x}_i}$ with the linear term $\sum_{i=1}^{t-1} \gamma_{ij}^{\top}\bm{x}_i$.}

Define\vspace{-0.15cm}
\begin{equation}
\hat{\bm{a}}_{tj}(\beta_{tj}) = \bar{\bm{a}}_{tj} + \beta_{tj}\frac{\psi_j}{\sqrt{n}}\bm{\gamma}_{tj},	
\end{equation}
\begin{equation}	\hat{\bm{A}}_t(\bm{\beta_t}) = (\hat{\bm{a}}_{t1}^\top(\beta_{t1}),\dots,\hat{\bm{a}}_{tm}^\top(\beta_{tm}))^\top.
\end{equation}
As shown in Algorithm \ref{Algo 2}, $\hat{\bm{A}}_t(\bm{\beta_t})$ is used in place of $\tilde{\bm{A}}_t$. 
That is, we use the expression $\sum_{t=1}^n \beta_{tj}\bm{\gamma}_{tj}^\top \bm{x}_t/\sqrt{n}$ to approximate $\sqrt{\sum_{t=1}^n \bm{x}_t^\top \bm{K}_{tj} \bm{x}_t}$. In round $t$, $\bm{\beta}_{t}$ is calculated \blue{according to \eqref{beta}} which is based on the historical decisions and will be used in the next iteration for correction.
It is worth noting that $\bm{\beta}_{t} = (\beta_{t1},\dots,\beta_{tm})^\top$ is calculated in each round and can be computed incrementally with low computational cost.

In the numerical experiments section \ref{exper}, it is illustrated that Algorithm \ref{Algo 2} has better performance than Algorithm \ref{Algo 1} in terms of the constraint violation. 
\blue{An intuitive explanation is that Algorithm \ref{Algo 2} is more inclined to reject the orders with high uncertainty of resource consumption (i.e., $\bm{K}_{tj}$) than Algorithm \ref{Algo 1} because $\bm \beta_t \ge \bm 1$.}

Next, we will propose another algorithm to correct the error (b). The error (b) consists of two parts, the optimality gap and constraint violation. 
% It is almost impossible to reduce the optimality gap and constraint violation simultaneously \red{(Is this sentence redundant?)}. 
In most cases, compared with the optimality gap, the CCP problems have a lower tolerance for the constraint violation, since the constraint violation will cause the probability
\begin{equation}
	\mathbb{P}_{\bm a_{tj}\sim \mathcal P_{a}}\left(\sum_{t=1}^n \bm{a}_{tj}^{\top} \bm{x}_t \le b_j \right)
\end{equation}
or the conditional expectations
\begin{equation}
	\mathbb{E}_{\bm a_{tj}\sim \mathcal P_{a}}\left[\frac{\sum_{t=1}^n\bm{a}_{tj}^{\top} \bm{x}_t - b_j}{\sqrt{\sum_{t=1}^n \bm{x}_t^\top \bm{K}_{tj} \bm{x}_t}}\bigg|\sum_{t=1}^n\bm{a}_{tj}^{\top} \bm{x}_t - b_j > 0\right]
\end{equation}
to deviate from their set values. For instance, if the confidence level of a chance constraint is set to 90\% but the actual confidence level of the online solution is 60\%, this online solution is unacceptable.

\blue{For the online LP problem,} Li et al.~\cite{li2020simple} have proposed a variant of the OPD algorithm that can reduce the amount of the constraint violation in practice. It is a non-stationary algorithm in which the resource consumption is considered while doing the subgradient descent.
% Specifically, the following two formulas (\ref{eq: non-stationary update1}) and (\ref{eq: non-stationary update2}) are used in place of Line 9 in Algorithm \ref{Algo 1}.
\blue{Specifically, instead of using the static average consumption $\bm{d}=\bm{b}/n$, the time-varying $\bm{d}_t$ calculated by the formula \eqref{eq: non-stationary update1} is utilized in Line 9 of Algorithm \ref{Algo 1}.}
\vspace{-0.15cm}
\begin{equation} \label{eq: non-stationary update1}
	\bm{d}_t = \frac{1}{n-t}\left(\bm{b}-\sum_{i=1}^t \tilde{\bm{A}}_{t}\bm{x}_t\right).\vspace{-0.1cm}
\end{equation}
% \begin{equation} \label{eq: non-stationary update2}
% \bm{p}_{t+1} = \max\left\{\bm{p}_{t} + \frac{1}{\sqrt{n}}\left(\tilde{\bm{A}}_{t}\bm{x}_t-\bm{d}_t\right), \mathbf{0}\right\}
% \end{equation}

\begin{algorithm}[t]\label{Algo 3}
%\SetAlgoNoLine
\caption{Modified Non-stationary OPD Algorithm for CCP}
\LinesNumbered
\KwIn{$\bm{d} = \bm{b}/n$}
\KwOut{$\bm{x}=(\bm{x}_1,...,\bm{x}_n)$}
{\bf Initialize:} $\bm{p}_1=\bm{0}$, $\bm{d}_1 = \bm{d}$\\
\For{$t=1,...,n$}
{
	Compute $\bm{\beta}_{t}$ via equation (\ref{beta})\\
	Set $v_t = \max_{l=1,\dots,k} \  (\bm{c}_t^\top - \bm{p}_t^\top \hat{\bm{A}}_{t}(\bm{\beta}_t))\bm{e}_l$\\
	\eIf{$v_t > 0$}
	{
		Pick an index $l_t$ randomly from\vspace{-0.15cm}
 $$\left\{l:v_t=(\bm{c}_t^\top - \bm{p}_t^\top \hat{\bm{A}}_{t}(\bm{\beta}_t))\bm{e}_l\right\}$$\\	\vspace{-0.15cm}
 		Set $\bm{x}_t = \bm{e}_{l_t}$
%		Set $$x_{tl}=\begin{cases}
%			1,\ & l = l_t\\
%			0,\ & \text{otherwise}
%		\end{cases}$$\\
	}{Set $\bm{x}_t = \bm{0}$}
	Compute $\bm{d}_{t}$ via equation \eqref{d_t}\\
	Compute 
	{\vspace{-0.15cm}
	$$\ \bm{p}_{t+1} = \max\left\{\bm{p}_{t} + \frac{1}{\sqrt{n}}\left(\hat{\bm{A}}_{t}(\bm{\beta}_t)\bm{x}_t-\bm{d}_t\right),\mathbf{0}\right\}$$
	\vspace{-0.25cm}
	}
}
\end{algorithm}

Inspired by the above non-stationary method, we propose the following formula to dynamically adjust the right-hand-side capacity $\bm{d}$ in each round:\vspace{-0.1cm}
\begin{equation}\label{d_t}
\begin{aligned}
	d_{tj} = & \frac{1}{n-t}\bigg(b_j - \sum_{i=1}^t \bar{\bm{a}}_{ij}^\top\bm{x}_i - \psi_j\sqrt{\frac{t}{n}\sum_{i=1}^t \bm{x}_i^\top \bm{K}_{ij} \bm{x}_i} \bigg),\forall j = 1,\dots,m,
\end{aligned}\vspace{-0.1cm}
\end{equation}
where 
\begin{equation*}\vspace{-0.1cm}
	\sum_{i=1}^t \bar{\bm{a}}_{ij}^\top\bm{x}_i + \psi_j\sqrt{\frac{t}{n}\sum_{i=1}^t \bm{x}_i^\top \bm{K}_{ij} \bm{x}_i} =\frac{t}{n}\left(\frac{n}{t}\sum_{i=1}^t \bar{\bm{a}}_{ij}^\top\bm{x}_i + \psi_j\sqrt{\frac{n}{t}\sum_{i=1}^t \bm{x}_i^\top \bm{K}_{ij} \bm{x}_i}\right)
\end{equation*}

\noindent is an estimation of the resource assumption at period $t$.

Then, we obtain Algorithm \ref{Algo 3} by merging the correction formula \eqref{d_t} into Algorithm \ref{Algo 2}.
% The correction formula \eqref{d_t} has the following effects:
\blue{The intuition behind the correction formula \eqref{d_t} is given as follows:} 
if too many resources are spent in the early rounds, the remaining resources will diminish. Then Algorithm \ref{Algo 3} will raise the dual price and be more likely to reject an order with high resource consumption as a result. On the other hand, if a large number of orders with high resource consumption are rejected at the start, resulting in an excess of remaining resources, Algorithm \ref{Algo 3} will decrease the dual price in order to accept more orders in the future. This \blue{correction strategy} makes Algorithm \ref{Algo 3} perform better than Algorithms \ref{Algo 1} and \ref{Algo 2} in numerical experiments.

% Moreover, the proposed algorithms are also effective for problems with both uncertain constraints and deterministic constraints. Taking Algorithm \ref{Algo 3} as an example, if constraint $j$ is deterministic, just replace $\hat{\bm{a}}_{tj}(\beta_{tj})$ with $\bm{a}_{tj}$ and do not update $d_j$.}

\section{Numerical Experiments}
\label{exper}

In this section, we present extensive numerical experiments to verify our algorithms. Several metrics are used to measure the performance of algorithms. In one trial, the calculation formula of these metrics are given as follows. $\{\bm x_{1}, \bm x_2,\dots,\bm x_n\}$ is the output of algorithms.

1) Probability deviation:
% the probability deviation of the whole problem in one trail is an average:
\begin{equation}\label{eq:prob_violation} \vspace{-0.2cm}
	\frac{1}{m}\sum_{j=1}^m\bigg(\eta_j-\Phi\bigg(\frac{b_j-\sum_{t=1}^n \bar{\boldsymbol{a}}_{tj}^\top\bm{x}_t}{\sqrt{\sum_{t=1}^n \boldsymbol{x}_t^\top \bm{K}_{tj} \boldsymbol{x}_t}}\bigg)\bigg)^+ \vspace{-0.1cm}
	\end{equation}
where $\bm{x}_t$ is the output of the algorithms. 

2) Optimality gap: 
%÷the optimality gap in one trail is calculated by 
\begin{equation}\label{experiment optimality gap} 
	\hat{R}_n^{SOCP}-\sum_{t = 1}^n \bm{c}_t^\top \bm{x}_t. \vspace{-0.1cm}
\end{equation}
The exact optimality gap should be calculated with $\hat{R}_n^{ISOCP}$. For ease of calculation, the offline optimal objective value $\hat{R}_n^{ISOCP}$ is approximated by $\hat{R}_n^{SOCP}$. Noting that $\hat{R}_n^{ISOCP}\le\hat{R}_n^{SOCP}$, the real optimality gap is \blue{upper bounded by} the value of \eqref{experiment optimality gap}.

3) Competitive ratio:\vspace{-0.1cm}
\begin{equation}
	\sum_{t = 1}^n \bm{c}_t^\top \bm{x}_t / \hat{R}_n^{SOCP} \times 100\%.
\end{equation}

4) Normalized constraint violation of conditional expectation constraints:
\begin{equation}\vspace{-0.1cm}
	\big|\big|(\tilde{\bm v}(\bm x))^+\big|\big|_2\vspace{-0.1cm}
\end{equation}
where
\begin{equation}\vspace{-0.2cm}
	\tilde{v}_j(\bm x) = \mathbb{E}_{\bm a_{tj}\sim \mathcal P_{a}}\left[\frac{\sum_{t=1}^n\boldsymbol{a}_{tj}^{\top} \boldsymbol{x}_t - b_j}{\sqrt{\sum_{t=1}^n \boldsymbol{x}_t^\top \bm{K}_{tj}\boldsymbol{x}_t}}\bigg|\sum_{t=1}^n\boldsymbol{a}_{tj}^{\top} \boldsymbol{x}_t > b_j \right] - \tilde{\gamma}_j.
\end{equation}  

5) Constraint violation of conditional expectation constraints:
\begin{equation}\vspace{-0.1cm}
	\big|\big|(\bm v(\bm x))^+\big|\big|_2
\end{equation}
where
$
	v_j(\bm x) = \tilde{v}_j(\bm x)\sqrt{\sum_{t=1}^n \boldsymbol{x}_t^\top \bm{K}_{tj}\boldsymbol{x}_t}.
$                                                                                                                                                                                                                                            

\subsection{CCP Problem with Chance Constraints (\ref{eq:cc}) Only} \label{sec:4.1}
In the first subsection, we apply the proposed algorithms to the classic CCP problem which does not contain the conditional expectation constraints \eqref{eq:ce}. Algorithm \ref{Algo 1}, Algorithm \ref{Algo 2}, Algorithm \ref{Algo 3} and Algorithm \ref{Algo 3} without correction (\ref{beta}) are compared in terms of optimality gap and probability deviation. 
These algorithms are implemented on two different models. 
Table \ref{tab:model} lists the distributions from which the elements in $\bm{c}_{tj}$, $\bm{\bar{a}}_{tj}$ or $\bm{K}_{tj}$ are i.i.d.~sampled.
$f(\chi^2(v))$ denotes $X = f(Y)$ and $Y\sim~\chi^2(v)$.
The uniform distribution is bounded while the chi-square distribution is unbounded. In other words, the CCP problems in Experiment I satisfy Assumption \ref{assumption 1} but the CCP problems in Experiment II do not. 
\begin{table}[ht]
    \vspace{-0.3cm}
	\small
	\renewcommand\arraystretch{1.35}
	\centering
	\caption{\small models used in the experiments.}
	\label{tab:model}
%	\resizebox{0.49\textwidth}{!}
	{
		\setlength{\tabcolsep}{15pt}{
			\begin{tabular}{c c c c c c c}
				\toprule[1.5pt]
				Experiment & $\bm{c}_{tj}$ & $\bm{\bar{a}}_{tj}$ & $\bm{K}_{tj}$ & $\bm{d}$\\
				\midrule[0.75pt]
				I & U[0, 1] & U[0, 4] & $(\text{U}[0, 1])^2$ & 1\\
				II & $\chi^2(3)$& $\frac{2}{3}\chi^2(4)$& $(\frac{2}{3}\chi^2(2))^2$ & 1\\
				\bottomrule[1.5pt]
			\end{tabular}
		}
	}
% 	\vspace{-0.3cm}
\end{table}

\subsubsection{Bounded Setting}\label{subsection: bounded setting}
In Experiment I, we set $k = 5$ and $m = 4$. 
The confidence levels of chance constraints are set to (0.65, 0.75, 0.85, 0.95). 
For each value of $n$, we run 20 simulation trials. 
In each trial, coefficients $\bm{c}_{tj}$, $\bm{\bar{a}}_{tj}$ and $\bm{K}_{tj}$ are resampled.
All metrics are averaged over all trials.

\begin{figure}[!b]
\vspace{-0.3cm}
\begin{minipage}[b]{.48\linewidth}
  \centering
  \centerline{\includegraphics[width=6.0cm]{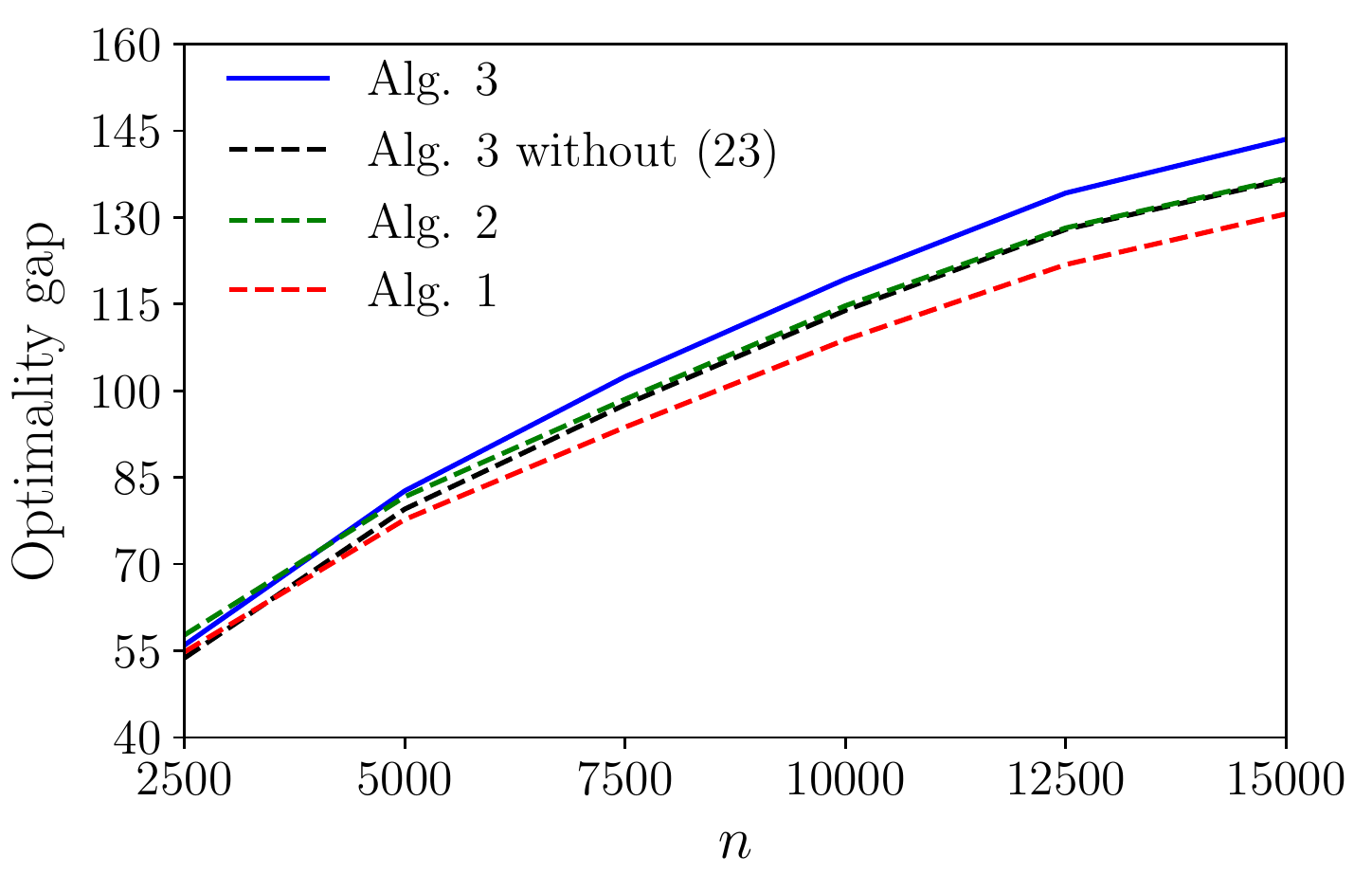}}
  \vspace{-0.3cm}
  \centerline{\small (a) optimality gap}\medskip
\end{minipage}
\hfill
\begin{minipage}[b]{0.48\linewidth}
  \centering
  \centerline{\includegraphics[width=6.0cm]{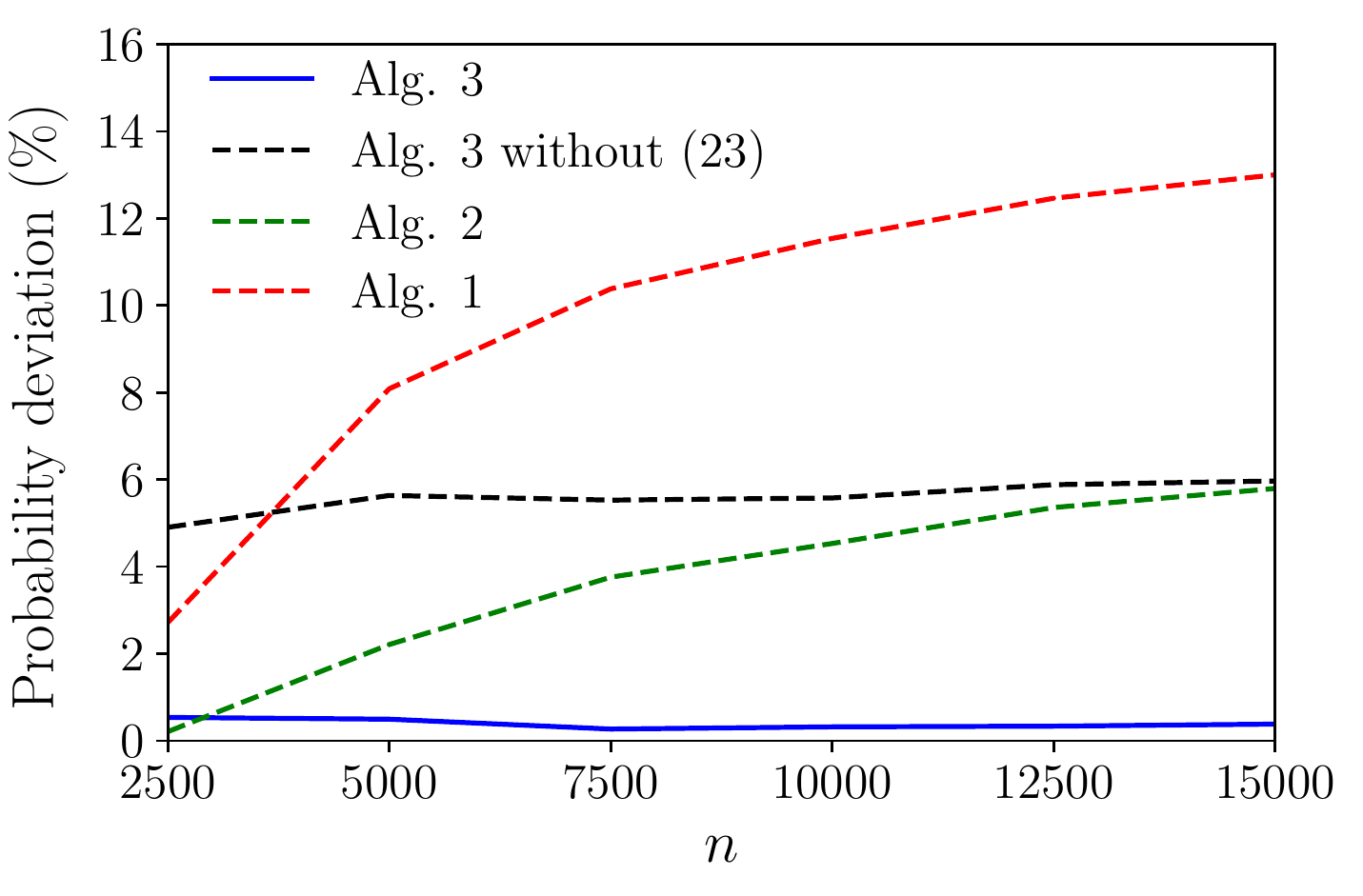}}
   \vspace{-0.3cm}
  \centerline{\small (b) probability deviation}\medskip
\end{minipage}
\vspace{-0.5cm}
\caption{\small optimality gap and probability deviation with uniform i.i.d. input.}
\vspace{0.3cm}
\label{fig:expr1_algo_comparison}
 \end{figure}

 \begin{figure}[htb]
\begin{minipage}[b]{.48\linewidth}
  \centering
  \centerline{\includegraphics[width=6.0cm]{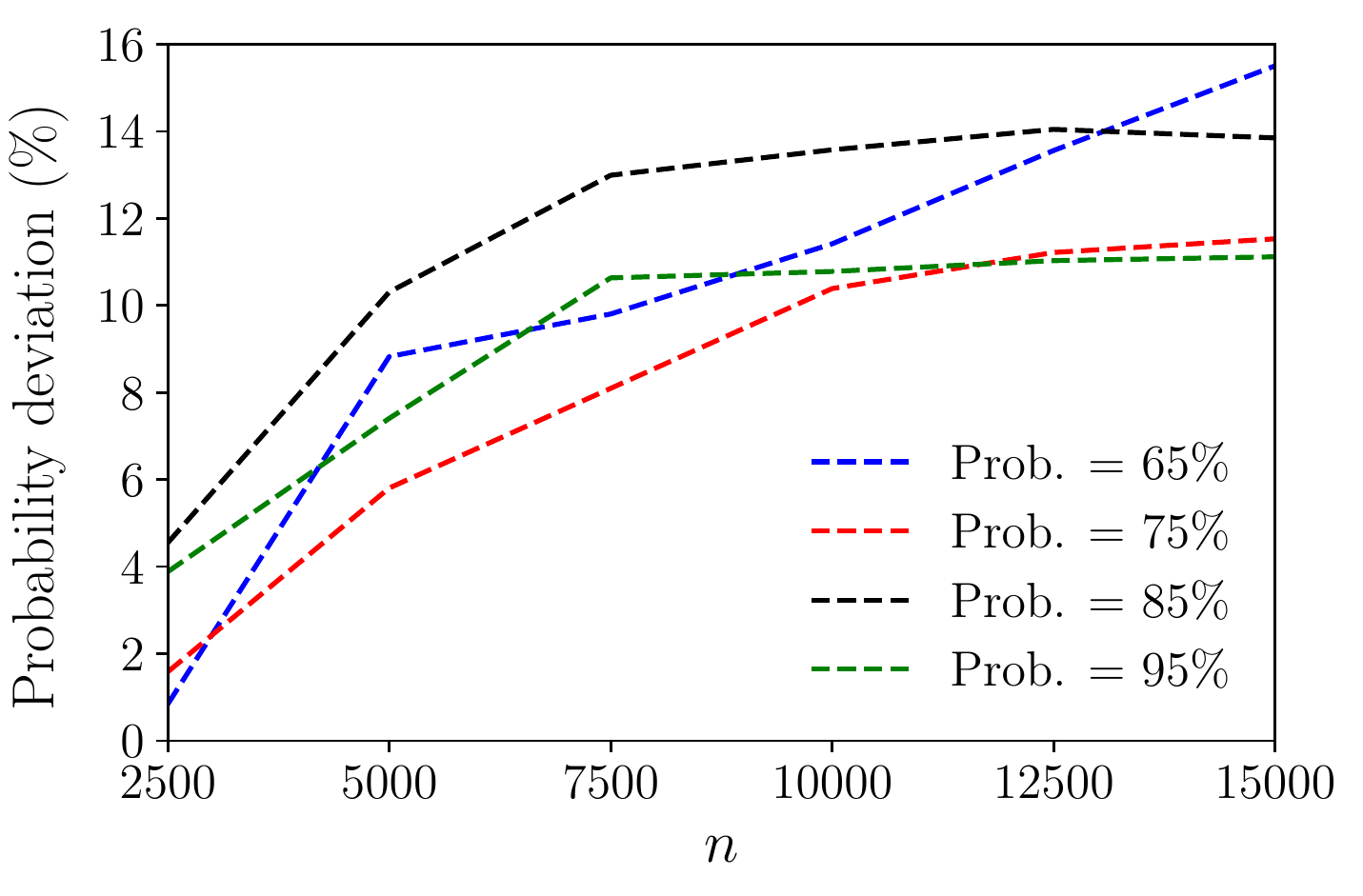}}
  \vspace{-0.3cm}
  \centerline{\small (a) Algorithm \ref{Algo 1}}\medskip
  \vspace{-0.2cm}
\end{minipage}
\hfill
\begin{minipage}[b]{0.48\linewidth}
  \centering
  \centerline{\includegraphics[width=6.0cm]{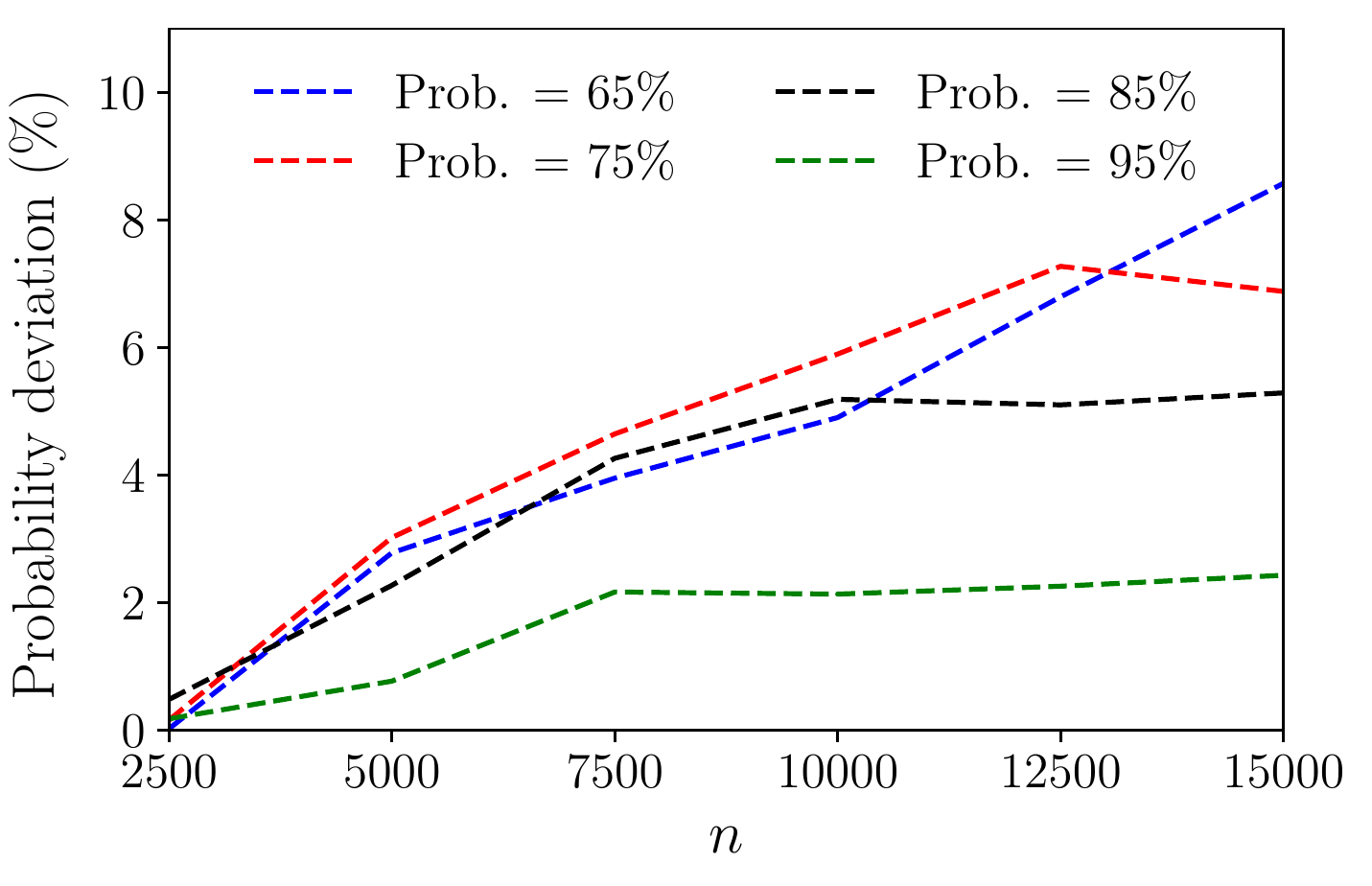}}
    \vspace{-0.3cm}
  \centerline{\small (b) Algorithm \ref{Algo 2}}\medskip
   \vspace{-0.2cm}
\end{minipage}
\begin{minipage}[b]{.48\linewidth}
  \centering
  \centerline{\includegraphics[width=6.0cm]{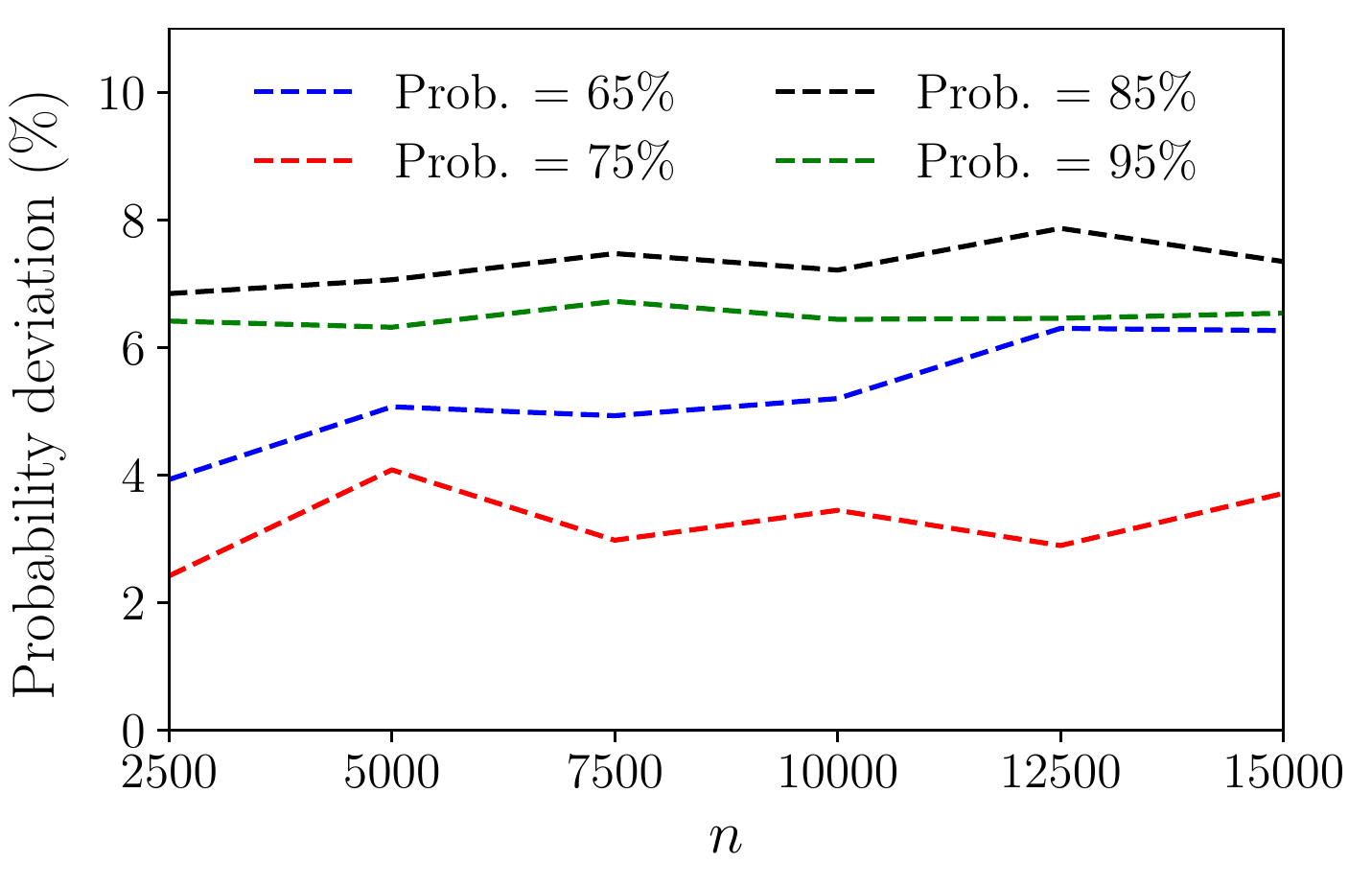}}
    \vspace{-0.3cm}
  \centerline{\small (c) Algorithm \ref{Algo 3} without correction (\ref{beta})}\medskip
\end{minipage}
\hfill
\begin{minipage}[b]{0.48\linewidth}
  \centering
  \centerline{\includegraphics[width=6.0cm]{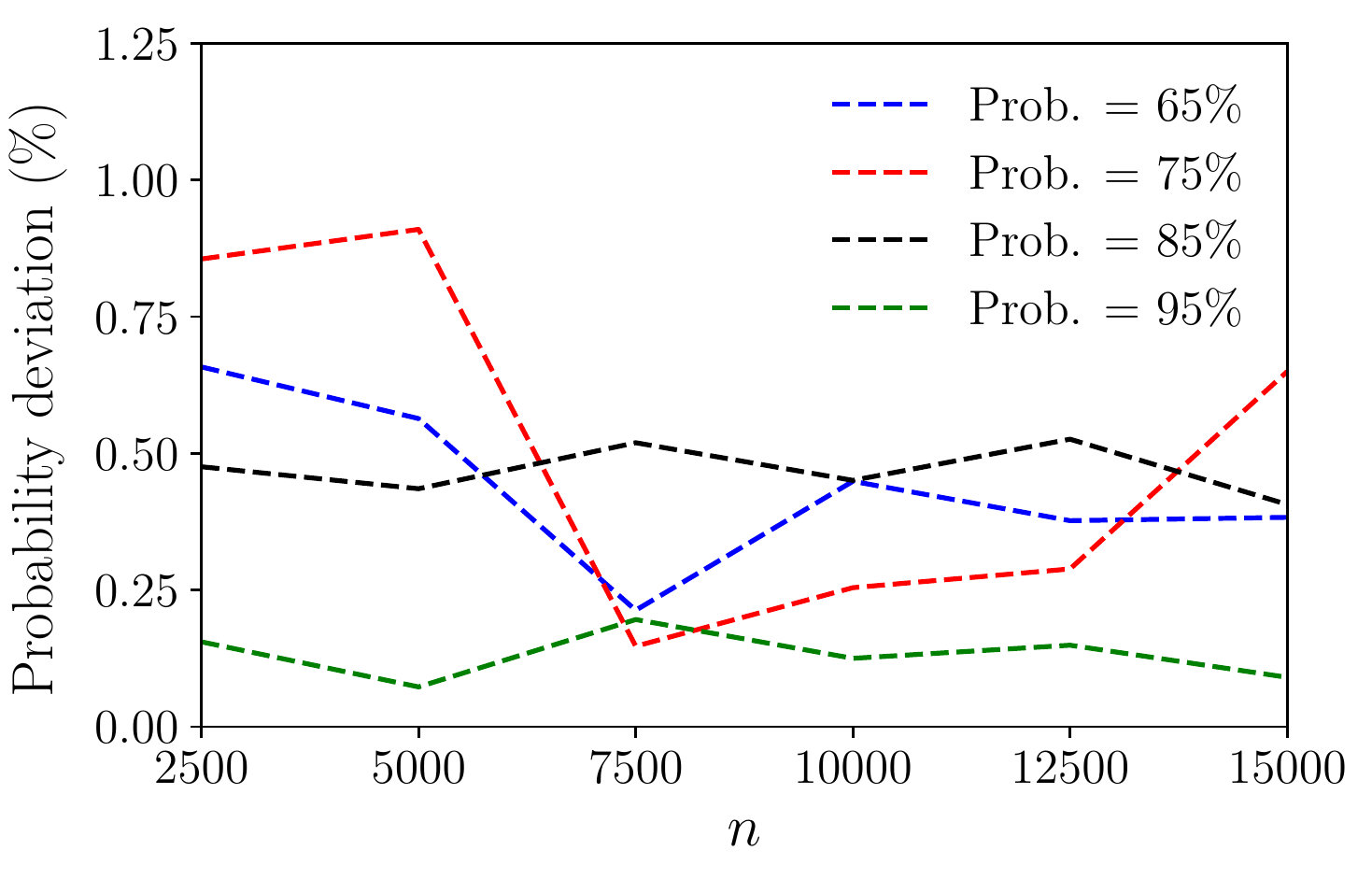}}
    \vspace{-0.3cm}
  \centerline{\small (d) Algorithm \ref{Algo 3}}\medskip
\end{minipage}
\vspace{-0.5cm}
\caption{\small probability deviation of each chance constraint with uniform i.i.d. input.}
% \vspace{-0.3cm}
\label{fig:expr1_detailed_violation}
\end{figure}

Figure\ \ref{fig:expr1_algo_comparison} shows the average optimality gap and probability deviation over all the simulation trials. 
From Figure~\ref{fig:expr1_algo_comparison}, we observe that the optimality gaps of these algorithms are very close, Algorithm \ref{Algo 3} has the smallest probability deviation and Algorithm \ref{Algo 1} has the largest probability deviation when $n \ge 5000$. 
For all $n$, Algorithm \ref{Algo 3} has all probability deviations less than 1\%.
Figure~\ref{fig:expr1_algo_comparison} (a) also shows that the optimality gaps of these four algorithms are on the order of $\sqrt{n}$ and verifies Theorem \ref{thm2}.
Figure~\ref{fig:expr1_detailed_violation} presents the probability deviations of each chance constraint of these four algorithms. For each chance constraint, the probability deviation of Algorithm \ref{Algo 3} is the smallest among the four algorithms when $n \ge 5000$.
Figure~\ref{fig:expr1_algo_comparison} (b) and Figure \ref{fig:expr1_detailed_violation} both illustrate that the proposed two corrections \eqref{beta} and \eqref{d_t} can effectively reduce the probability deviation with \blue{minor negative effects} on the optimality gap.

\subsubsection{Unbounded Setting}
In Experiment II, $k$ and $m$ are still set to 5 and 4. The confidence levels are also the same as those in Experiment I. For each value of $n$, we run 20 simulation trials.
In each trial, coefficients $\bm{c}_{tj}$, $\bm{\bar{a}}_{tj}$ and $\bm{K}_{tj}$ are i.i.d.\ sampled from the chi-square distributions which are unbounded.

\begin{figure}[!b]
\vspace{-0.3cm}
\begin{minipage}[b]{.48\linewidth}
  \centering
  \centerline{\includegraphics[width=6.0cm]{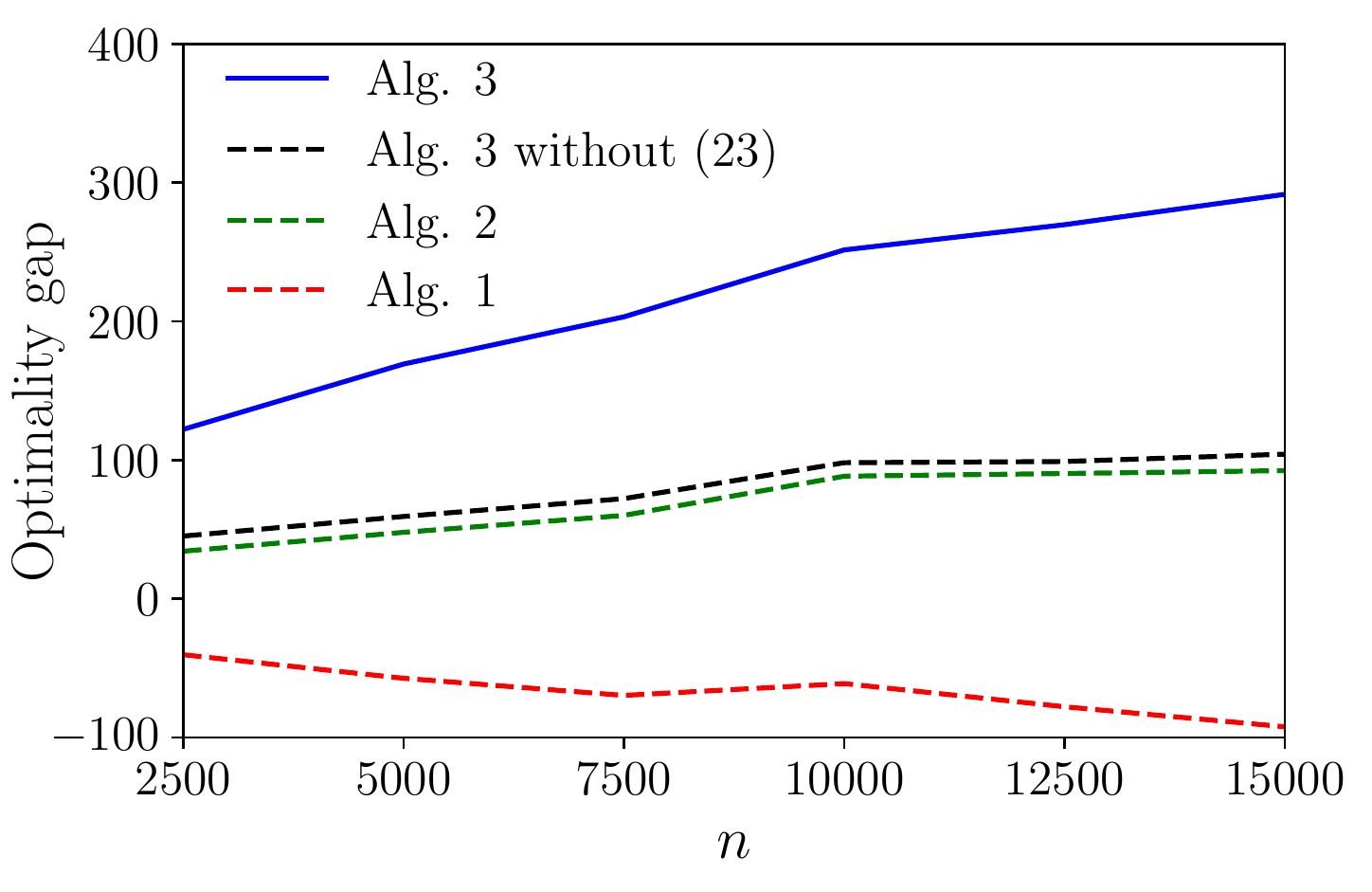}}
    \vspace{-0.2cm}
  \centerline{\small (a) optimality gap}\medskip
\end{minipage}
\hfill
\begin{minipage}[b]{0.48\linewidth}
  \centering
  \centerline{\includegraphics[width=6.0cm]{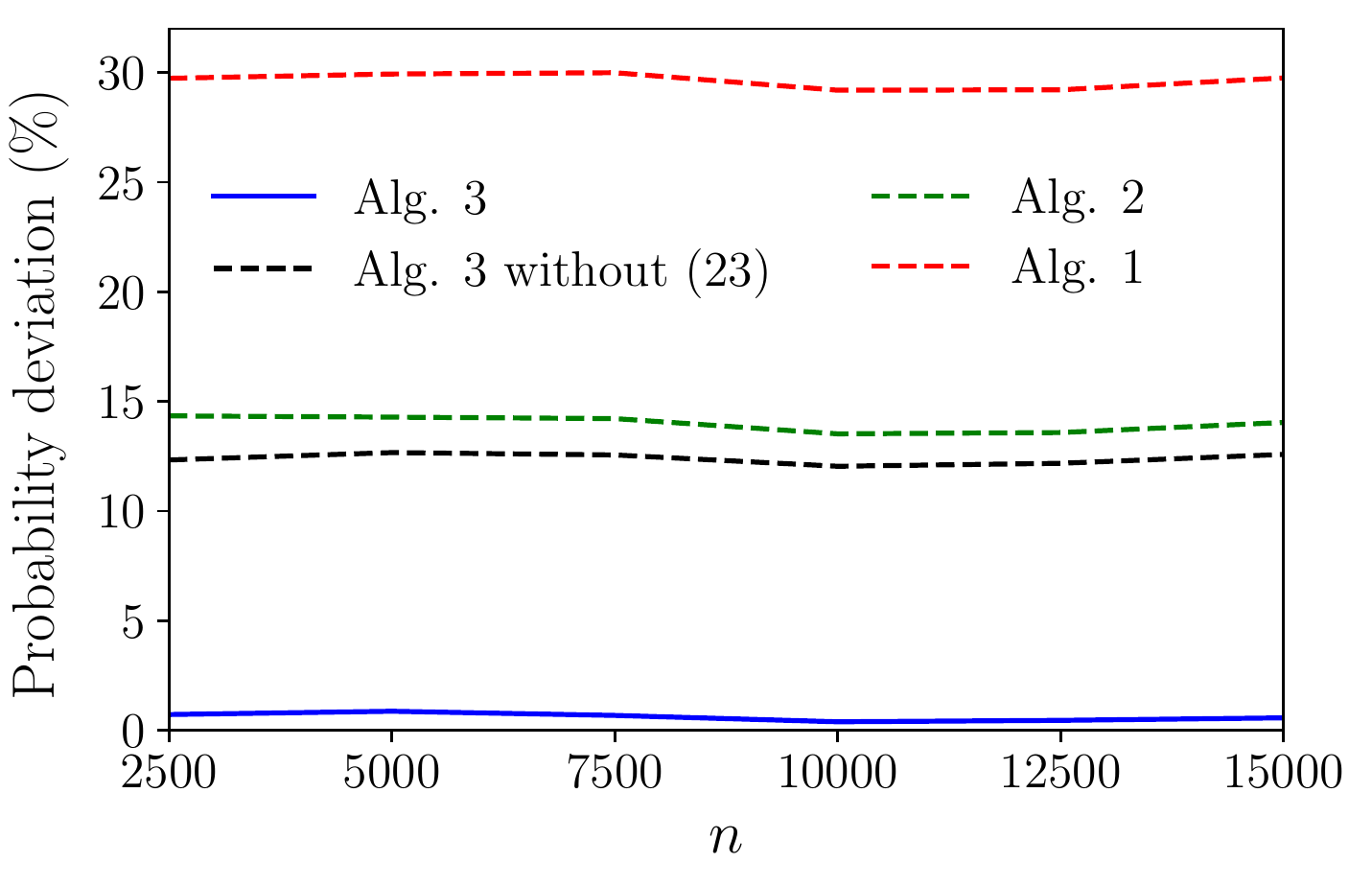}}
    \vspace{-0.2cm}
  \centerline{\small (b) probability deviation}\medskip
\end{minipage}
\vspace{-0.5cm}
\caption{\small optimality gap and probability deviation with chi-square i.i.d. input.}
\vspace{-0.3cm}
\label{fig:expr2_algo_comparison}
\end{figure}

Figure~\ref{fig:expr2_algo_comparison} shows the average optimality gap and probability deviation, and Figure~\ref{fig:expr2_detailed_violation} shows the detailed probability deviations of each chance constraint.
The results of Experiment II are similar to those of Experiment I: 
Algorithm \ref{Algo 3} has the smallest probability deviation, whereas Algorithm \ref{Algo 1} has the largest probability deviation.
Experiment II again verifies that corrections (\ref{beta}) and (\ref{d_t}) can effectively reduce the probability deviation.
\blue{Despite the fact that} Algorithm \ref{Algo 3} produces slightly larger optimality gap, its optimality gap is still approximately on the order of $\sqrt{n}$. 
In this experiment with unbounded input, Algorithm \ref{Algo 3} significantly reduces the probability deviations:
the probability deviations of the algorithms except Algorithm \ref{Algo 3} are larger than 10\%, while the probability deviation of Algorithm \ref{Algo 3} is less than 1\%.

\begin{figure}[!h]
\begin{minipage}[b]{.48\linewidth}
  \centering
  \centerline{\includegraphics[width=6.0cm]{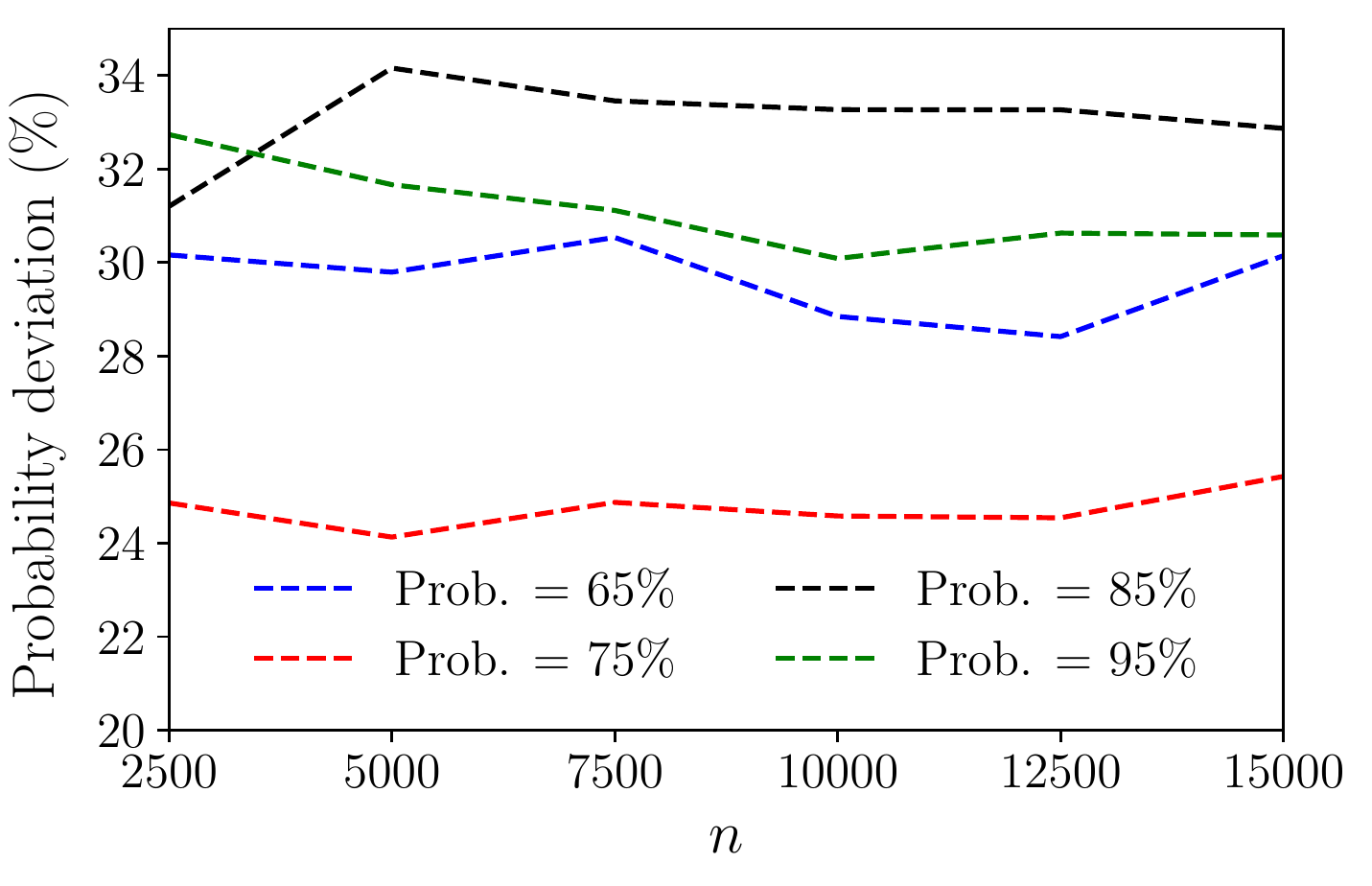}}
    \vspace{-0.2cm}
  \centerline{\small (a) Algorithm \ref{Algo 1}}\medskip
   \vspace{-0.2cm}
\end{minipage}
\hfill
\begin{minipage}[b]{0.48\linewidth}
  \centering
  \centerline{\includegraphics[width=6.0cm]{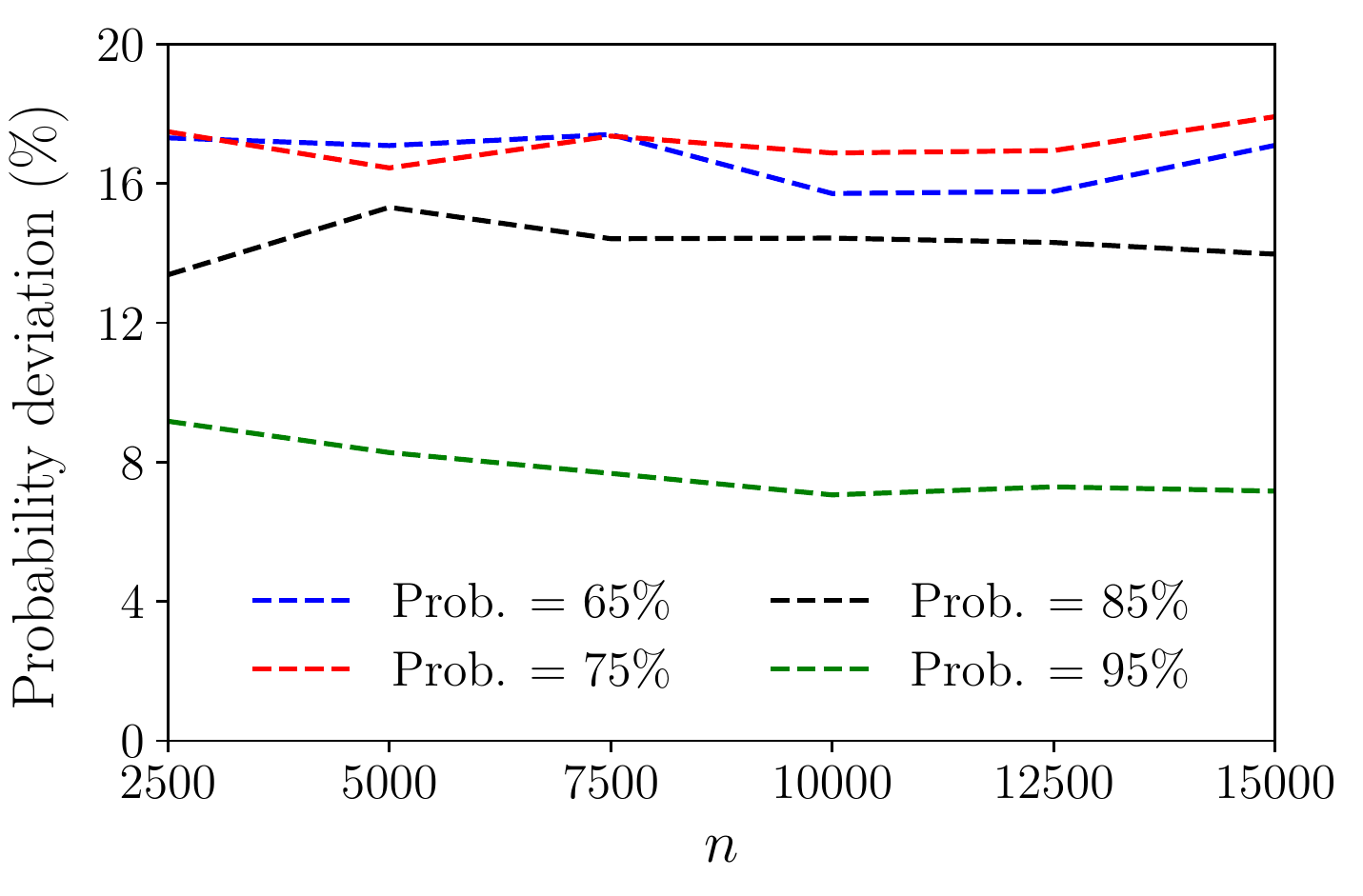}}
    \vspace{-0.2cm}
  \centerline{\small (b) Algorithm \ref{Algo 2}}\medskip
   \vspace{-0.2cm}
\end{minipage}
\begin{minipage}[b]{.48\linewidth}
  \centering
  \centerline{\includegraphics[width=6.0cm]{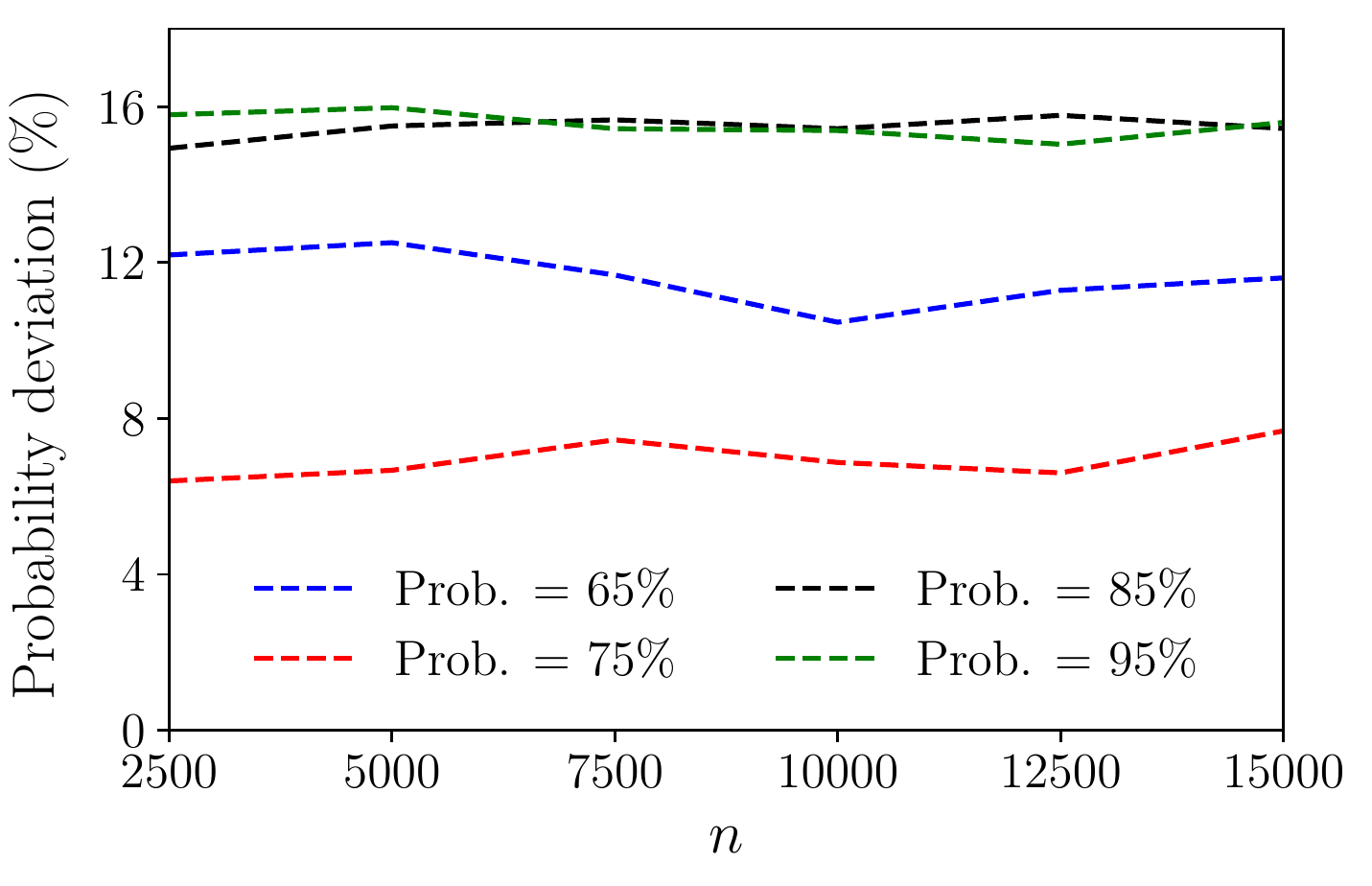}}
    \vspace{-0.2cm}
  \centerline{\small (c) Algorithm \ref{Algo 3} without correction (\ref{beta})}\medskip
\end{minipage}
\hfill
\begin{minipage}[b]{0.48\linewidth}
  \centering
  \centerline{\includegraphics[width=6.0cm]{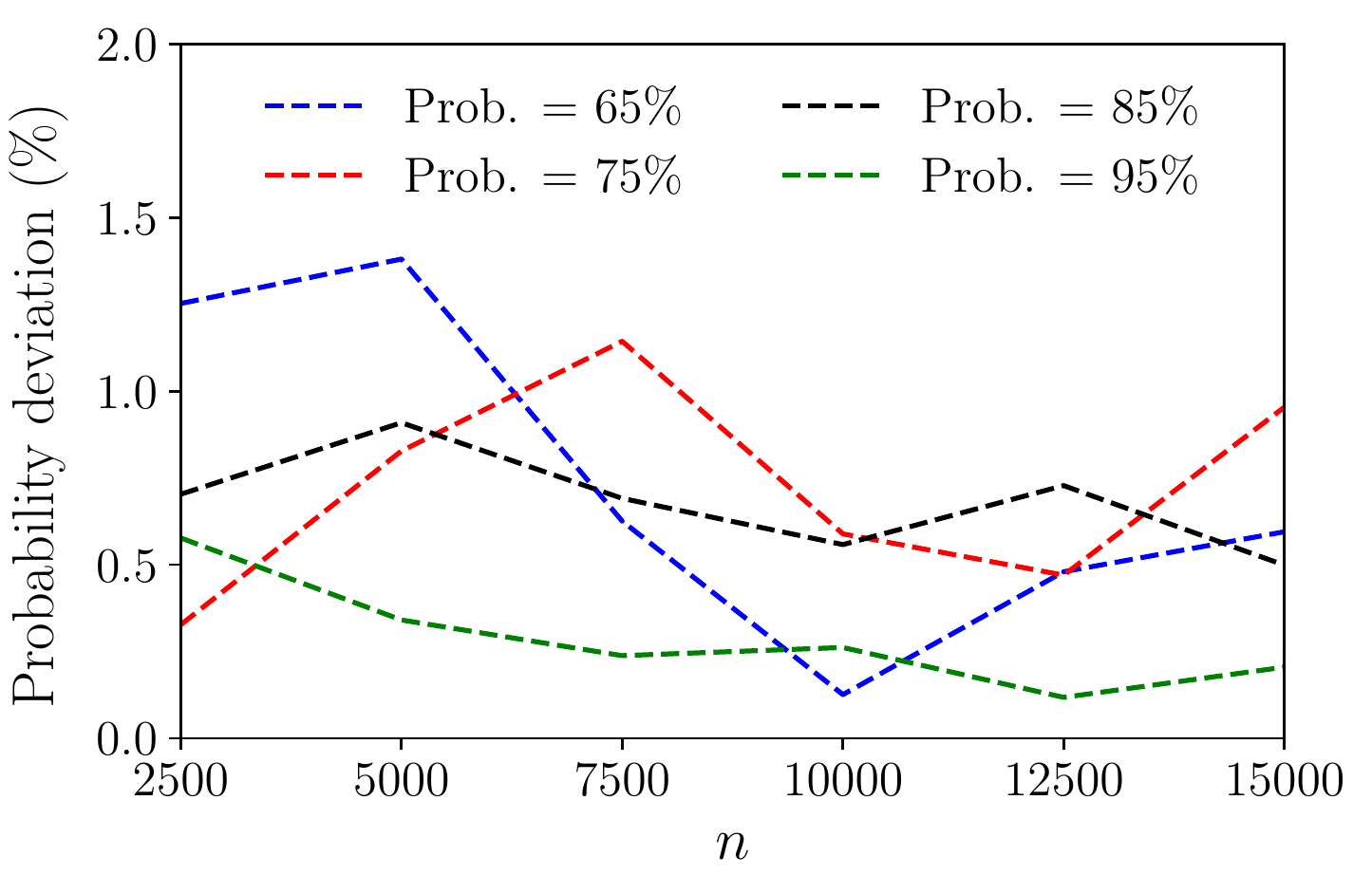}}
    \vspace{-0.2cm}
  \centerline{\small (d) Algorithm \ref{Algo 3}}\medskip
\end{minipage}
\vspace{-0.3cm}
\caption{\small probability deviation of each chance constraint with chi-square i.i.d. input.}
\vspace{-0.3cm}
\label{fig:expr2_detailed_violation}
\end{figure}
\begin{table}[!h]
	\small
	\renewcommand\arraystretch{1.35}
	\centering
	\caption{\small competitive ratios of Algorithm \ref{Algo 3} in the experiments.}
	\label{tab:competitive ratio}
%	\resizebox{0.49\textwidth}{!}
	{
		\setlength{\tabcolsep}{5pt}{
			\begin{tabular}{p{60pt}<{\centering}p{30pt}<{\centering}p{30pt}<{\centering}
			p{30pt}<{\centering}p{30pt}<{\centering}p{30pt}<{\centering}p{30pt}<{\centering}}
				\toprule[1.5pt]
				\multirow{2}{*}{Experiment} & \multicolumn{6}{c}{Number of requests $n$} \\
				\cmidrule[0.75pt]{2-7}
				~ & 2500 & 5000 & 7500 & 10000 & 12500 & 15000\\
				\midrule[0.75pt]
				I & 95.90\% & 97.13\% & 97.69\% & 97.94\% & 98.14\% & 98.38\%\\
				II & 98.67\%& 99.09\%& 99.27\% & 99.32\% &99.42\% & 99.48\% \\
				\bottomrule[1.5pt]
			\end{tabular}
		}
	}
\end{table}

The competitive ratios of Algorithm 2 in Experiment I and II are larger than 95\%, with details provided in Table \ref{tab:competitive ratio}.

\subsection{CCP Problem with Conditional Expectation Constraints (\ref{eq:ce}) Only}
In this subsection, we apply the proposed algorithms to the CCP problem with only conditional expectation constraints and compare the performance in terms of the normalized constraint violation and the constraint violation. The input data are the same as these in Section \ref{sec:4.1}. The parameters $k$ and $m$ are also set to 5 and 4, and $\tilde{\gamma}_j$ in four conditional expectation constraints are set to 0.2, 0.3, 0.4 and 0.5, respectively. The scale of $\tilde{\gamma}_j$ matches the standard normal distribution: 0.2, 0.3, 0.4 and 0.5 correspond to 0.2, 0.3, 0.4 and 0.5 times the standard deviation, respectively. 
% Thus, the setting of this experiment is reasonable.

\begin{figure}[!htb]
\vspace{-0.2cm}
\begin{minipage}[b]{.48\linewidth}
  \centering
  \centerline{\includegraphics[width=6.0cm]{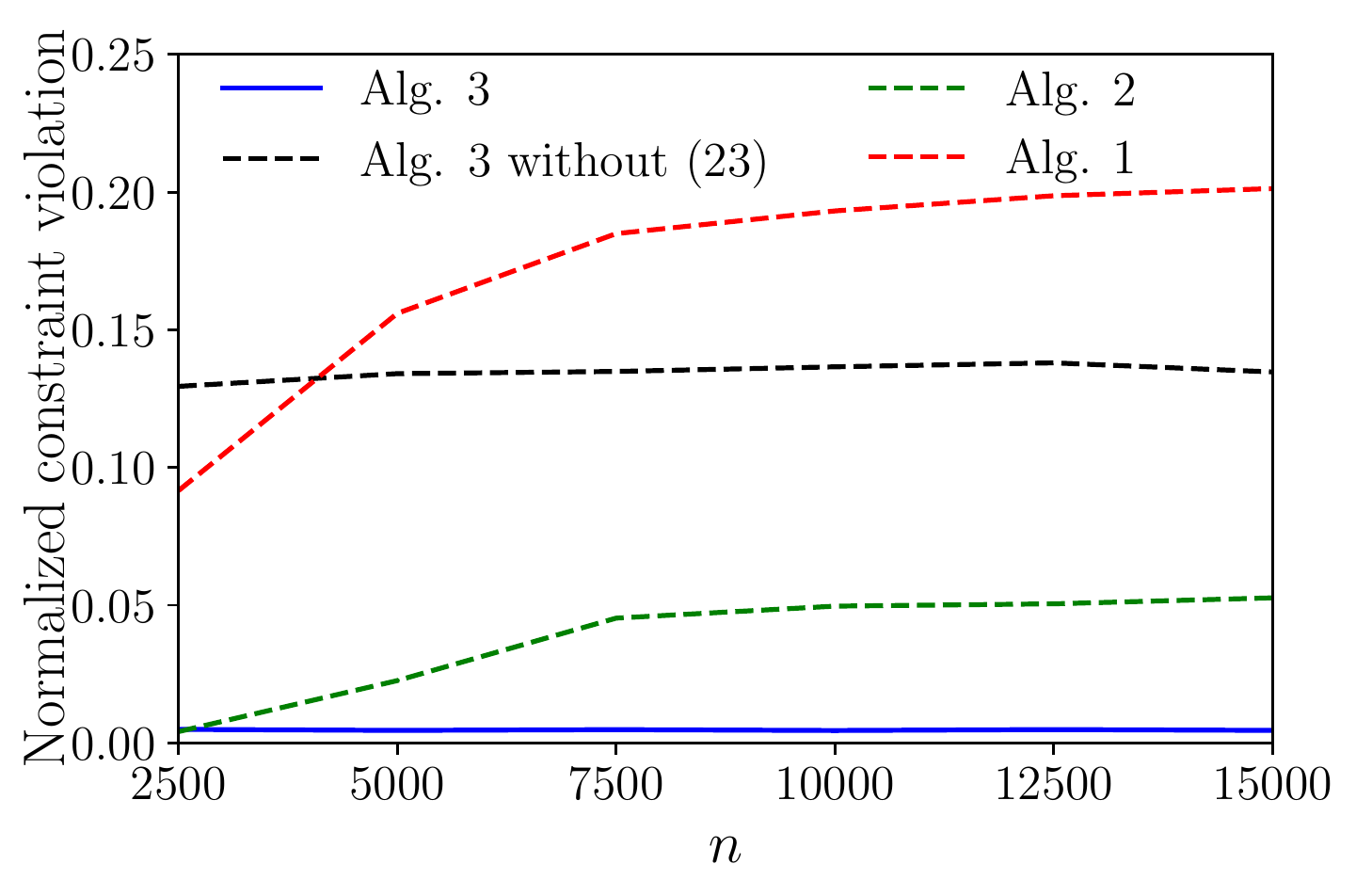}}
  \vspace{-0.15cm}
  \centerline{\small (a) normalized constraint violation}\medskip
\end{minipage}
\hfill
\begin{minipage}[b]{0.48\linewidth}
  \centering
  \centerline{\includegraphics[width=6.0cm]{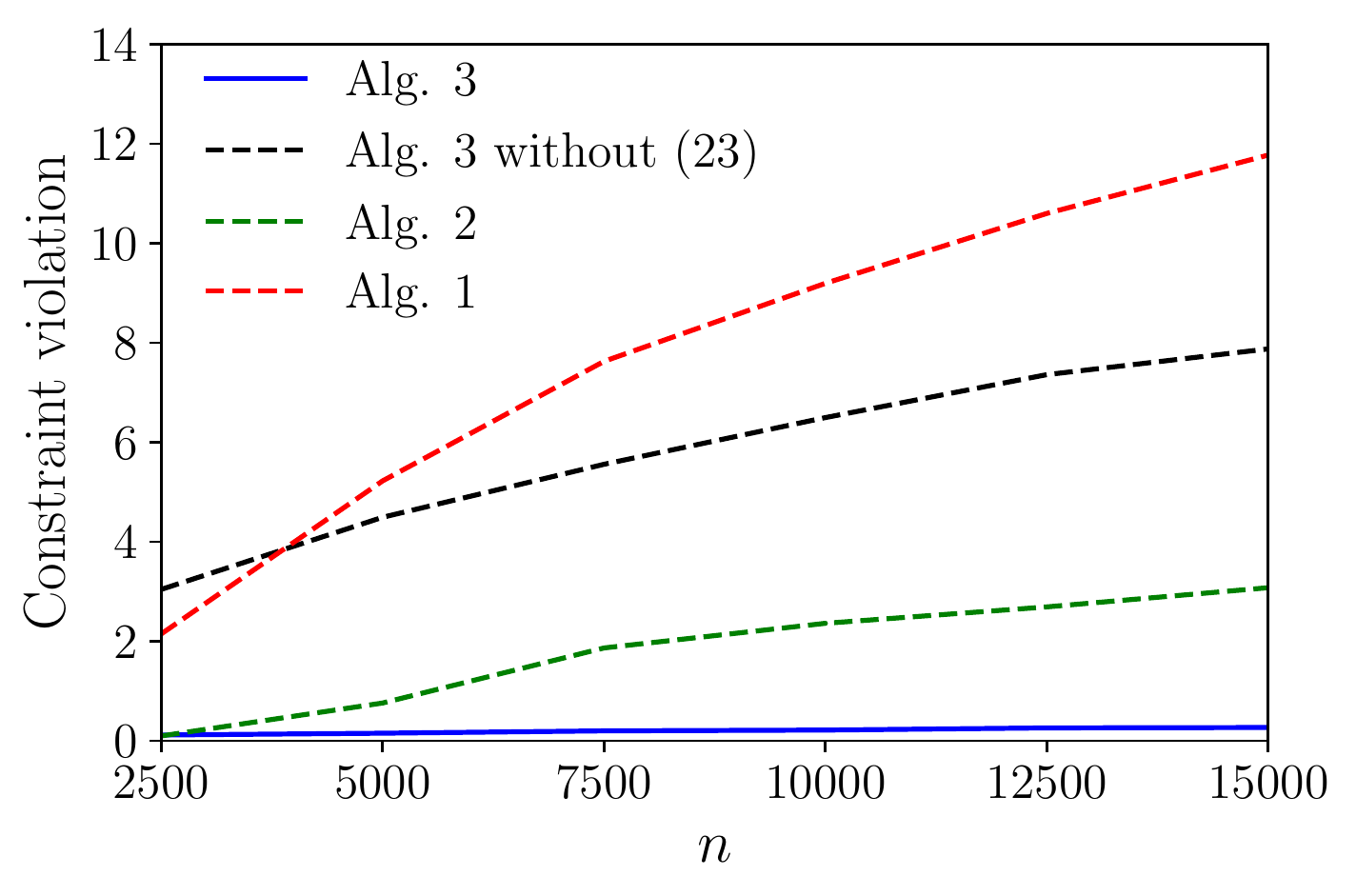}}
  \vspace{-0.15cm}
  \centerline{\small (b) constraint violation}\medskip
\end{minipage}
\vspace{-0.5cm}
\caption{normalized constraint violation and constraint violation with uniform i.i.d. input.}

\label{fig:cec_comparison_norm}

\begin{minipage}[b]{.48\linewidth}
  \centering
  \centerline{\includegraphics[width=6.0cm]{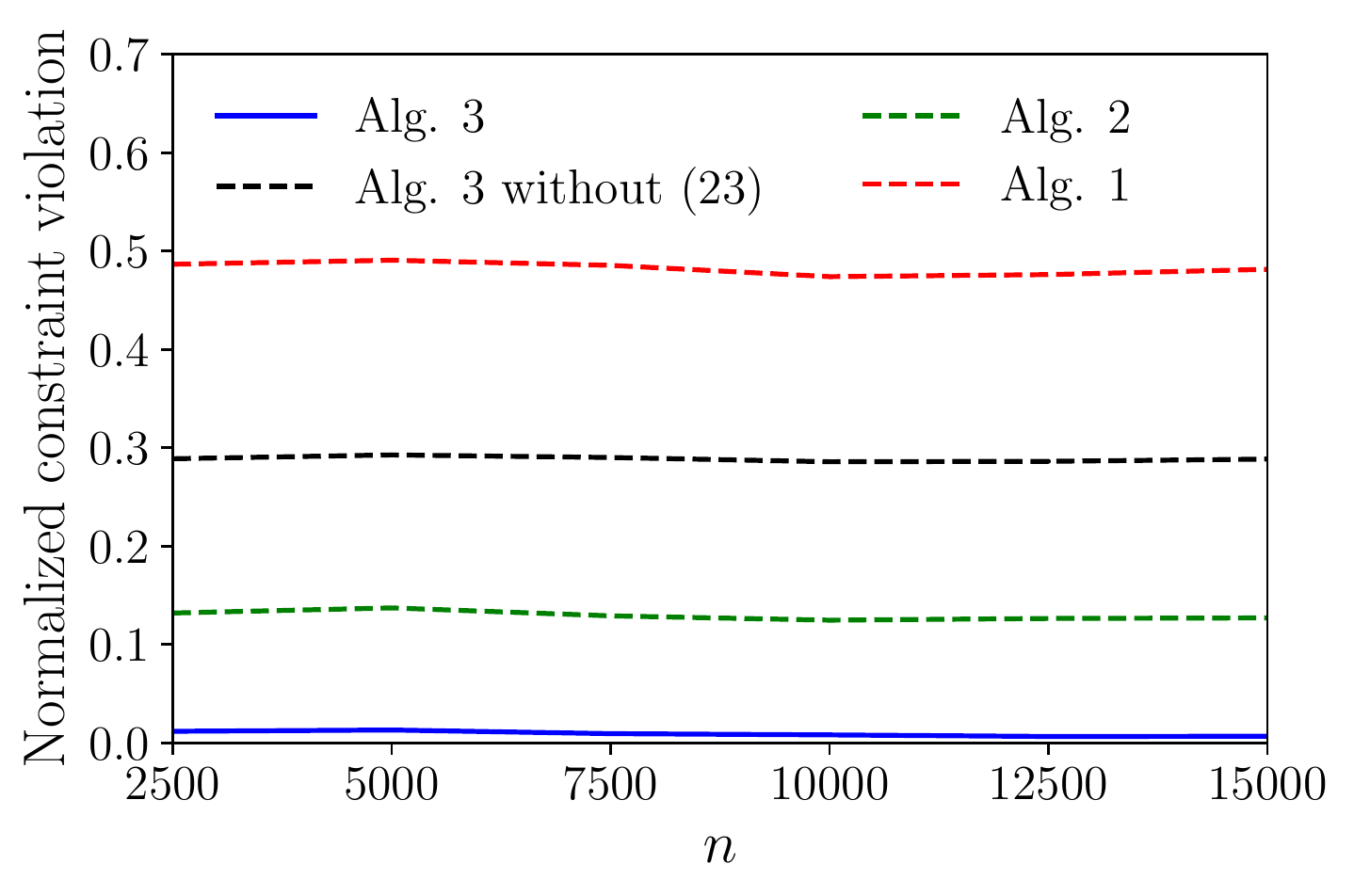}}
  \vspace{-0.15cm}
  \centerline{\small (a) normalized constraint violation}\medskip
\end{minipage}
\hfill
\begin{minipage}[b]{0.48\linewidth}
  \centering
  \centerline{\includegraphics[width=6.0cm]{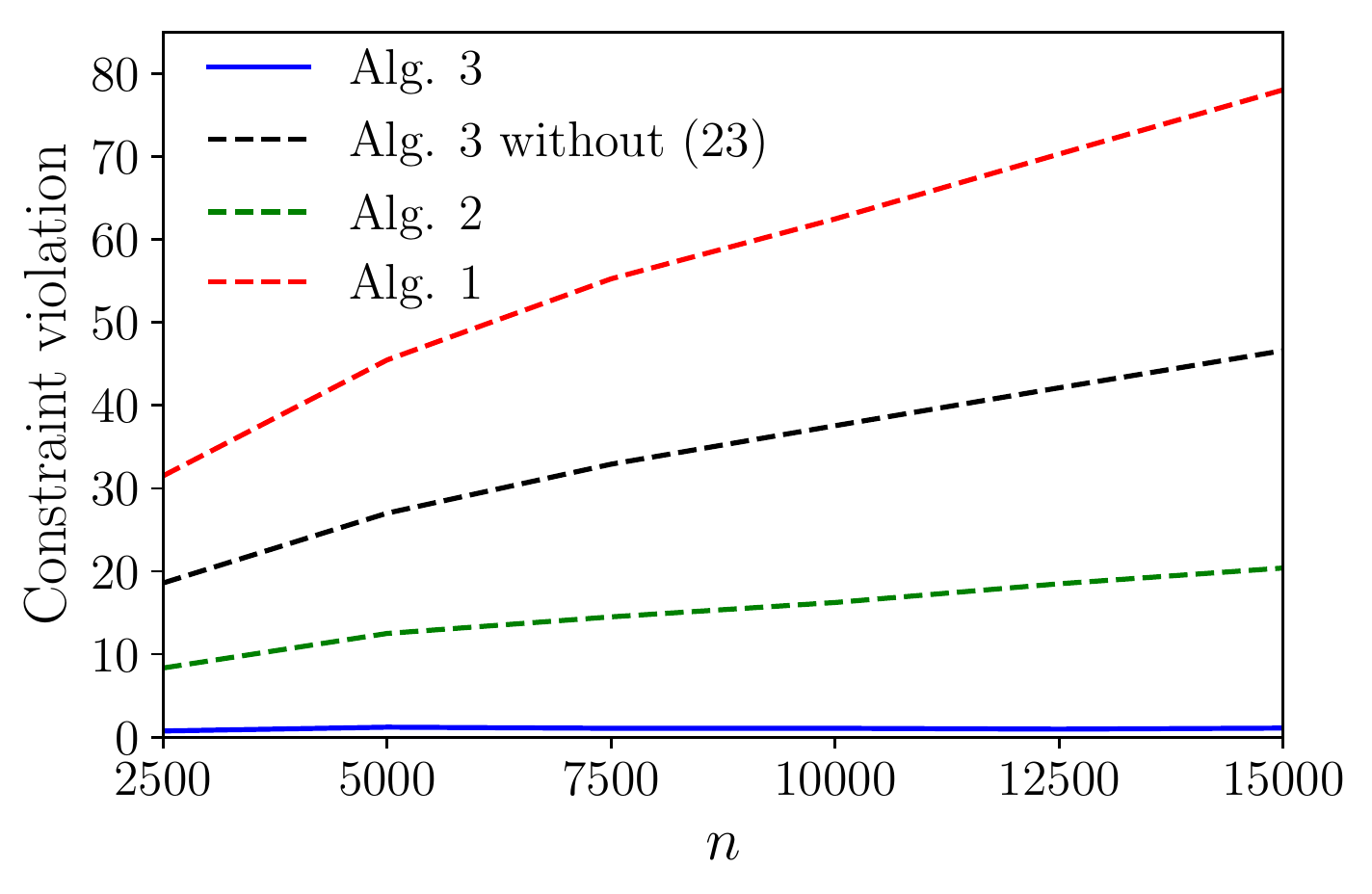}}
  \vspace{-0.15cm}
  \centerline{\small (b) constraint violation}\medskip
\end{minipage}
\vspace{-0.5cm}
\caption{normalized constraint violation and constraint violation with chi-square i.i.d. input.}
\vspace{-0.3cm}

\label{fig:cec_comparison}
\end{figure}

Figure \ref{fig:cec_comparison_norm} and Figure \ref{fig:cec_comparison} show the normalized constraint violation and constraint violation of four algorithms in the cases with uniform and chi-square inputs.
In both scenarios, Algorithm \ref{Algo 3} has the smallest normalized constraint violation and constraint violation while Algorithm \ref{Algo 1} has the largest violations. It is verified that the proposed corrections can effectively reduce the violation of conditional expectation constraints. Moreover, these experimental results show that the constraint violation of four algorithms in both cases is on the order of $\sqrt{n}$. This order is the same as stated in Theorem \ref{thm2}. Note that Theorem \ref{thm2} only analyzes the performance of Algorithm \ref{Algo 1} and its analysis relies on the boundedness assumption of the input. The theoretical analysis of conditional expectation constraints is still an open question.

%\subsection{Comparison of Step Size}
%\input{sections/experiment4.3.tex}

%\subsection{Real-Data Experiment}
%\input{sections/experiment4.4.tex}

\section{Conclusion}
\label{conclusion}
In this paper, we study the online stochastic RAP with chance constraints and conditional expectation constraints under mild assumptions.
First, we present a linearization method that decouples the non-linear term in second-order cone constraints and makes the online solution possible.
Next, we theoretically analyze the performance of the vanilla OPD algorithm compared with the offline solution of the stochastic RAP. Assuming the input data is bounded and i.i.d.~sampled, the expected optimality gap and constraint violation are both $O(\sqrt{n})$.
After that, we propose modified OPD algorithms with several heuristic corrections to reduce the probability derivation and constraint violation.
Numerical results verify the effectiveness of proposed heuristics and reveal that the probability derivation and constraint violation of the proposed Algorithm \ref{Algo 3} is very small. 
Additionally, numerical results confirm that Algorithm \ref{Algo 3} can output a near-optimal solution with small regret.
In conclusion, the proposed Algorithm \ref{Algo 3} is suitable for solving online CCP and can be applied to engineering applications such as online order fulfillment.

\bibliographystyle{plain} 
\bibliography{cas-refs}

\appendix
\section{Proof of Proposition \ref{prop:1}}
\label{sec:appendix prop 1}
First, $\{\bm{x} \in \{0,1\}^k|\bm{1}^\top\bm{x} \le 1\}$ = $\{\bm e_s \in \mathbb R^k|s=1,\dots,k\}$, where $\bm e_l$ is the unit vector with $l$-th coordinate being 1.

Then, for all $\bm e_l \in \{\bm e_s \in \mathbb R^k|s=1,\dots,k\}$, we have
\begin{equation}
    \sqrt{\bm{e}_l^\top \bm{K}_{tj} \bm{e}_l} = \sqrt{\bm K_{tjll}} = \bm \gamma_{tj}^\top \bm e_l,
\end{equation}
where $\bm K_{tjll}$ is $l$-th diagonal element of the matrix $\bm K_{tj}$ and $\bm \gamma_{tj} = (\sqrt{\bm K_{tj11}},\dots,$ $\dots,\sqrt{\bm K_{tjkk}})^\top$.

Hence, Proposition \ref{prop:1} holds.

\section{Proof of Theorem \ref{thm: lower bound of OPD}}
\label{sec:appendix thm lower bound opt gap}
In this section, we denote the output of Algorithm \ref{Algo 1} as $\bm {\tilde x}_t$ to distinguish it from the decision variables $\bm x_t$.

First, we present the following proposition.

\begin{prop}\label{prop:strong convex}
    Under Assumption \ref{assumption 1}, the problems \eqref{Linear-relaxedCCP} and \eqref{Linear-NLCCP} have strong duality.
\end{prop}
\begin{proof}
Setting $\bm x_t = \bm 0, \forall t=1,\dots,n$, we have
\begin{equation}
\begin{aligned}
    &\sum_{t = 1}^n \left(\bar{\bm{a}}_{tj}^\top + \frac{\psi_j}{\sqrt{n}} \bm{\gamma}^\top_{tj}  \right) \bm{x}_t = 0 \le b_j,\forall j = 1,\dots,m\\
    &\sum_{t=1}^n \bar{\boldsymbol{a}}_{tj}^\top \boldsymbol{x}_t + \psi_j\sqrt{\sum_{t=1}^n \boldsymbol{x}_t^\top \bm{K}_{tj} \boldsymbol{x}_t} = 0 < b_j, \forall j = 1,\dots,m\\
    &\bm 1^\top\bm x_t \le 1, \bm x_t \ge \bm 0, \forall t = 1,\dots,n
\end{aligned}
\end{equation}
which means that Slater's conditions of \eqref{Linear-relaxedCCP} and \eqref{Linear-NLCCP} are satisfied. Thus, strong duality holds.
\end{proof}

{\noindent \bf{Proof of Theorem \ref{thm: lower bound of OPD}}}
\vspace{0.4cm}

Based on the strong duality of linear problem \eqref{Linear-relaxedCCP}, we have 
\begingroup
\allowdisplaybreaks
\begin{align}
	\hat{R}_n^{LP} =&\min_{\bm p \ge 0} \max_{\bm 1^\top \bm x \le 1, \bm x \ge \bm 0} \sum_{j=1}^m p_j b_j + \sum_{t=1}^n (\bm c_t^\top - \sum_{j=1}^m p_j(\bar {\bm a}_{tj}^\top+\frac{\psi_j}{\sqrt{n}} \bm \gamma_{tj}^\top))\bm{x}_t\notag \\
	=& \max_{\bm 1^\top \bm x \le 1, \bm x \ge \bm 0}\min_{\bm p \ge 0} \sum_{j=1}^m p_j b_j + \sum_{t=1}^n (\bm c_t^\top - \sum_{j=1}^m p_j(\bar {\bm a}_{tj}^\top+\frac{\psi_j}{\sqrt{n}} \bm \gamma_{tj}^\top))\bm{x}_t\notag \\
	= & \max_{\bm 1^\top \bm x \le 1, \bm x \ge \bm 0} \sum_{j=1}^m \hat{p}_j^{LP} b_j + \sum_{t=1}^n (\bm c_t^\top - \sum_{j=1}^m \hat{p}_j^{LP}(\bar {\bm a}_{tj}^\top + \frac{\psi_j}{\sqrt{n}} \bm \gamma_{tj}^\top))\bm{x}_t \\
	\ge & \sum_{j=1}^m \hat{p}_j^{LP} b_j + \sum_{t=1}^n (\bm c_t^\top - \sum_{j=1}^m \hat{p}_j^{LP}(\bar {\bm a}_{tj}^\top+\frac{\psi_j}{\sqrt{n}} \bm \gamma_{tj}^\top))\tilde{\bm{x}}_t\notag\\
	=& \sum_{t=1}^n \bm c_t^\top \tilde{\bm x}_t + \sum_{j=1}^m \hat{p}_j^{LP}(b_j - (\bar {\bm a}_{tj}^\top+\frac{\psi_j}{\sqrt{n}} \bm \gamma_{tj}^\top)\tilde{\bm x}_t)\notag
\end{align}
where $\hat{p}_j^{LP}$ is the optimal dual value of \eqref{Linear-relaxedCCP}. The second equation holds because of strong duality \cite{boyd2004convex}. 

Next, since $\hat p_j^{LP}$ is $O(1)$ \cite{li2019online} and 
    \begin{align}
	& \mathbb{E}\left[\sum_{j=1}^m\left(b_j - (\bar {\bm a}_{tj}^\top+\frac{\psi_j}{\sqrt{n}}\sum_{t=1}^n \bm \gamma_{tj}^\top)\tilde{\bm x}_t\right)\right]\\
	\ge & -\mathbb{E}\left[\left\|\left(\sum_{t=1}^n \tilde{\bm{A}}_{t} \tilde{\bm{x}}_t-\bm{b}\right)^+\right\|_1\right]
	\ge -\mathbb{E}\left[\sqrt{m}\left\|\left(\sum_{t=1}^n \tilde{\bm{A}}_{t} \tilde{\bm{x}}_t-\bm{b}\right)^+\right\|_2\right]\notag \\
	\ge & -O(\sqrt{n})\notag,
\end{align}
where the last inequality comes from Theorem \ref{thm1},
we get
\begin{equation}
	\mathbb{E}\left[\hat{R}_n^{LP}-\sum_{t=1}^n\bm c_t^\top \tilde{\bm x}_t\right] \ge -O(\sqrt{n}).
\end{equation}

\section{Proof of Lemma \ref{lemma1}}
\label{sec:appendix lemma1}

Consider the dual problem of the SOCP problem (\ref{Linear-NLCCP}):
\begin{equation}\label{prob:dual SOCP}
\begin{aligned}
	\min_{\bm p \ge 0} \max_{\bm 1^\top  \bm x \le 1, \bm x \ge \bm 0} \sum_{j=1}^m p_j b_j + \sum_{t=1}^n (\bm c_t^\top - \sum_{j=1}^m p_j\bar {\bm a}_{tj}^\top)\bm{x}_t -\sum_{j=1}^m p_j\psi_j\sqrt{\sum_{t=1}^n \bm{x}_t^\top \bm{K}_{tj} \bm{x}_t}.
\end{aligned}
\end{equation}
The optimal value of $\bm p$ has the following property.

\begin{prop}\label{prop: dual value boundedness}
	Under Assumption \ref{assumption 1}, the optimal solution of the problem (\ref{prob:dual SOCP}) is bounded:
\begin{equation}
	\hat{\bm p}^{SOCP} \in \left\{\bm p \in \mathbb R^m \bigg|\bm p \ge \bm0, \bm 1^\top \bm p \le \frac{\bar{c}}{\underline{d}}\right\}
\end{equation}
where $\hat{\bm p}^{SOCP}$ is the optimal solution, $\bar c$ is the upper bound of $c_{tl}$ and $\underline d$ is the lower bound of $d_j$, which means $\bar c \ge c_{tl}, \ \forall t = 1,\dots,n,\  l=1,\dots,k$ and $d_j = b_j/n \ge \underline d > 0, \ \forall j = 1,\dots,m$.
\end{prop}
\begin{proof}
Use $f(\bm p)$ to denote the value of dual function
\begin{equation}
	f(\bm p):=\max_{\bm 1^\top \bm x \le 1, \bm x \ge \bm 0} \sum_{j=1}^m p_j b_j + \sum_{t=1}^n (\bm c_t^\top - \sum_{j=1}^m p_j\bar {\bm a}_{tj}^\top)\bm{x}_t
 -\sum_{j=1}^m p_j\psi_j\sqrt{\sum_{t=1}^n \bm{x}_t^\top \bm{K}_{tj} \bm{x}_t},
\end{equation}
and use $L(\bm p, \bm x)$ to denote the value of Lagrangian function

\begin{equation}
	L(\bm p, \bm x):=\sum_{j=1}^m p_j b_j + \sum_{t=1}^n (\bm c_t^\top - \sum_{j=1}^m p_j\bar {\bm a}_{tj}^\top)\bm{x}_t
	 -\sum_{j=1}^m p_j\psi_j\sqrt{\sum_{t=1}^n \bm{x}_t^\top \bm{K}_{tj} \bm{x}_t}.
\end{equation}	

	We only need to prove $\bm d^\top \hat{\bm p}^{SOCP} \le \bar c$. If $\bm d^\top \hat{\bm p}^{SOCP} > \bar c$, then
\begin{equation}
\begin{aligned}
	f(\hat{\bm p}^{SOCP})& \ge L(\hat{\bm p}^{SOCP}, \bm0) = \sum_{j=1}^m \hat{p}_j^{SOCP}b_j \\
	& =n\bm d^\top \hat{\bm p}^{SOCP} > n\bar c\\
	& \ge \max_{\bm 1^\top \bm x \le 1, \bm x \ge \bm 0}\sum_{t=1}^n \bm c_t^\top \bm x_t = f(\bm 0) 
\end{aligned}
\end{equation}
which contradicts that $\hat{\bm p}^{SOCP}$ is the optimal solution of the dual problem \eqref{prob:dual SOCP} that is $\min_{\bm p \ge \bm 0}\ f(\bm p)$. 

Hence, $\bm d^\top \hat{\bm p}^{SOCP} \le \bar c$. Evidently, $\bm 1^\top \hat{\bm p}^{SOCP} \le \bar{c} / \underline{d}$.
\end{proof}

{\noindent \bf{Proof of Lemma \ref{lemma1}}}
\begin{proof}
	(a) For the optimality gap, $\hat{R}_n^{SOCP} \le \sum_{t=1}^n c_t^\top \hat{\bm x}_t^{LP}$. Specifically, 
  \begin{align}\label{B6}
	\hat{R}_n^{SOCP} =&\min_{\bm p \ge 0} \max_{\bm 1^\top \bm x \le 1, \bm x \ge \bm 0} \sum_{j=1}^m p_j b_j +\notag\\
	&\sum_{t=1}^n (\bm c_t^\top - \sum_{j=1}^m p_j\bar {\bm a}_{tj}^\top)\bm{x}_t -\sum_{j=1}^m p_j\psi_j\sqrt{\sum_{t=1}^n \bm{x}_t^\top \bm{K}_{tj} \bm{x}_t}\notag\\
	\le & \min_{\bm p \ge 0} \max_{\bm 1^\top \bm x \le 1, \bm x \ge \bm 0} \sum_{j=1}^m p_j b_j+\\
	& \sum_{t=1}^n (\bm c_t^\top - \sum_{j=1}^m p_j\bar {\bm a}_{tj}^\top)\bm{x}_t-\sum_{j=1}^m p_j\frac{\psi_j}{\sqrt{n}}\sum_{t=1}^n \bm \gamma_{tj}^\top \bm{x}_t \notag\\
	 =& \sum_{t=1}^n\bm c_t^\top \hat{\bm x}_t^{LP}\notag.
\end{align}  

The first equation is from strong duality of \eqref{Linear-NLCCP} (Proposition \ref{prop:strong convex}).

Next is to prove $\hat{R}_n^{SOCP} \ge\sum_{t=1}^n c_t^\top \hat{\bm x}_t^{LP} - O(\sqrt{n})$.
We have \vspace{-0.1cm}
\begingroup
\allowdisplaybreaks
\begin{align}\label{B7}
	\hat{R}_n^{SOCP}= & \max_{\bm 1^\top \bm x \le 1, \bm x \ge \bm 0} \sum_{j=1}^m \hat{p}_j^{SOCP} b_j + \sum_{t=1}^n (\bm c_t^\top - \sum_{j=1}^m \hat{p}_j^{SOCP}\bar {\bm a}_{tj}^\top)\bm{x}_t \notag \\
	&-\sum_{j=1}^m \hat{p}_j^{SOCP}\psi_j\sqrt{\sum_{t=1}^n \bm{x}_t^\top \bm{K}_{tj} \bm{x}_t} \notag\\
	\ge & \max_{\bm 1^\top \bm x \le 1, \bm x \ge \bm 0} \sum_{j=1}^m \hat{p}_j^{SOCP} b_j + \sum_{t=1}^n (\bm c_t^\top - \sum_{j=1}^m \hat{p}_j^{SOCP}\bar {\bm a}_{tj}^\top)\bm{x}_t \notag\\
	&-\sum_{j=1}^m \hat{p}_j^{SOCP}\left(\frac{\psi_j}{\sqrt{n}}\sum_{t=1}^n \sqrt{\bm{x}_t^\top \bm{K}_{tj} \bm{x}_t} +\psi_j\sqrt{\sum_{t=1}^n \bm{x}_t^\top \bm{K}_{tj} \bm{x}_t}\right)\\
	\ge & \max_{\bm 1^\top \bm x \le 1, \bm x \ge \bm 0} \sum_{j=1}^m \hat{p}_j^{SOCP} b_j + \sum_{t=1}^n (\bm c_t^\top - \sum_{j=1}^m \hat{p}_j^{SOCP}\bar {\bm a}_{tj}^\top)\bm{x}_t \notag\\
	&-\sum_{j=1}^m \hat{p}_j^{SOCP}\frac{\psi_j}{\sqrt{n}}\sum_{t=1}^n \bm \gamma_{tj}^\top \bm{x}_t - 
	\frac{\bar c}{\underline{d}}\bar{\psi} \bar \gamma \sqrt{n}\notag\\
	\ge &\max_{\bm 1^\top \bm x \le 1, \bm x \ge \bm 0} \sum_{j=1}^m \hat{p}_j^{LP} b_j + \sum_{t=1}^n (\bm c_t^\top - \sum_{j=1}^m \hat{p}_j^{LP}\bar {\bm a}_{tj}^\top)\bm{x}_t \notag\\
	&-\sum_{j=1}^m \hat{p}_j^{LP}\frac{\psi_j}{\sqrt{n}}\sum_{t=1}^n \bm \gamma_{tj}^\top \bm{x}_t -\frac{\bar c}{\underline{d}}\bar{\psi} \bar \gamma \sqrt{n}\notag\\
	= & \sum_{t=1}^n c_t^\top \hat{\bm x}_t^{LP} - O(\sqrt{n})\notag
\end{align}
\endgroup
where $\bar{\gamma} = \sqrt{\bar K}$, $\bar K$ is the upper bound of the diagonal elements of $\bm K_{tj}$, $\bar \psi = \max_{j}{\psi_j}$ and $\hat{\bm p}^{LP}$ is the optimal solution of the dual problem of (\ref{Linear-relaxedCCP}). In (\ref{B7}), the second inequality comes from Proposition \ref{prop:1}, Proposition \ref{prop: dual value boundedness} and the boundedness of $\bm K_{tj}$. The third inequality is valid because the dual problem \vspace{-0.1cm}
\begin{equation}
	\min_{\bm p \ge \bm 0}\max_{\bm 1^\top \bm x \le 1, \bm x \ge \bm 0} \sum_{j=1}^m p_j b_j + \sum_{t=1}^n (\bm c_t^\top - \sum_{j=1}^m p_j (\bar {\bm a}_{tj}^\top+\frac{\psi_j}{\sqrt{n}}\bm \gamma_{tj}^\top ))\bm{x}_t
\end{equation}
achieves its minimum value when $\bm p = \hat{\bm p}^{LP}$.

According to (\ref{B6}) and (\ref{B7}), 
\begin{equation}
	\bigg|\hat{R}_n^{SOCP}-\sum_{t = 1}^n \bm{c}_t^\top \hat{\bm{x}}_t^{LP}\bigg| \le O(\sqrt{n})
\end{equation}
holds for all $n$.

(b) For the resource consumption gap, we have 
\begin{equation}
\begin{aligned}
	g_j(\bm x)-\sum_{t=1}^n \tilde{\bm{a}}_{tj}\bm x_t & = \psi_j\sqrt{\sum_{t=1}^n \boldsymbol{x}_t^\top \bm{K}_{tj} \boldsymbol{x}_t} - \sum_{t=1}^n \frac{\psi_j}{\sqrt{n}}\bm \gamma_{tj}^\top \bm{x}_t\\
	& \le \psi_j\sqrt{\sum_{t=1}^n \boldsymbol{x}_t^\top \bm{K}_{tj} \bm{x}_t} \le \psi_j \sqrt{\bar{K}}\sqrt{n}.
\end{aligned}	
\end{equation}
Hence,
\begin{equation}
    \left\|\bm g\left({\bm{x}}\right)-\sum_{t=1}^n\tilde{\bm A}_t \bm x_t\right\|_2 \le O(\sqrt{n})
\end{equation}
holds for all $n$.
\end{proof}

\section{Proof of Theorem \ref{thm2}}
\label{sec:appendix thm2}
\begin{proof}
	According to Lemma \ref{lemma1} and Theorem \ref{thm1}, we have
\begin{equation}
\begin{aligned}
	& ~\mathbb E \left[\hat{R}_n^{SOCP}-\sum_{t = 1}^n \bm{c}_t^\top \bm{x}_t\right] \\
	= &~\mathbb E \left[\hat{R}_n^{SOCP}-\sum_{t = 1}^n \bm{c}_t^\top \hat{\bm{x}}_t^{LP}\right]
	+ \mathbb E \left[\sum_{t = 1}^n \bm{c}_t^\top \hat{\bm{x}}_t^{LP}-\sum_{t = 1}^n \bm{c}_t^\top \bm{x}_t\right]\\
	\le &~ 0 + O(\sqrt{n}) =  O(\sqrt{n}).
\end{aligned}
\end{equation}

According to Lemma \ref{lemma1} and Therorem \ref{thm: lower bound of OPD}, we have
\begin{align}
	& ~\mathbb E \left[\hat{R}_n^{SOCP}-\sum_{t = 1}^n \bm{c}_t^\top \bm{x}_t\right]\notag \\
	= &~\mathbb E \left[\hat{R}_n^{SOCP}-\sum_{t = 1}^n \bm{c}_t^\top \hat{\bm{x}}_t^{LP}\right]
	+ \mathbb E \left[\sum_{t = 1}^n \bm{c}_t^\top \hat{\bm{x}}_t^{LP}-\sum_{t = 1}^n \bm{c}_t^\top \bm{x}_t\right]\\
	\ge &~ -O(\sqrt{n}) - O(\sqrt{n}) = -O(\sqrt{n})\notag.
\end{align}

For the constraint violation, we have
\begin{equation}
\begin{aligned}
 & ~{\mathbb E} \left[\left\|\left(\bm g\left({{\bm{x}}}\right)-\bm{b}\right)^+\right\|_2\right]	\\
 \le & ~{\mathbb E} \left[\left\|\left(\bm g\left({{\bm{x}}}\right)-\sum_{t=1}^n\tilde{\bm{A}}_t\bm x_t\right)^+\right\|_2\right]	
 + {\mathbb E} \left[\left\|\left(\sum_{t=1}^n\tilde{\bm{A}}_t\bm x_t-\bm{b}\right)^+\right\|_2\right]\\
 \le & ~{\mathbb E} \left[\left\|\bm g\left({{\bm{x}}}\right)-\sum_{t=1}^n\tilde{\bm{A}}_t\bm x_t\right\|_2\right]	
 + {\mathbb E} \left[\left\|\left(\sum_{t=1}^n\tilde{\bm{A}}_t\bm x_t-\bm{b}\right)^+\right\|_2\right]	\\
 \le & ~O(\sqrt n) + O(\sqrt n) = O(\sqrt n).
\end{aligned}
\end{equation}
Thus we obtain Theorem \ref{thm2}.
\end{proof}

\end{document}